\documentclass[11pt]{article}
\usepackage[a4paper,margin=29mm]{geometry}
\usepackage{amsmath,amssymb,amsthm,mathtools,mathrsfs,array,booktabs,longtable}
\usepackage[hidelinks]{hyperref}
\usepackage{microtype}
\usepackage[all]{xy}

\newtheorem{theorem}{Theorem}[section]
\newtheorem{proposition}[theorem]{Proposition}
\newtheorem{lemma}[theorem]{Lemma}

\theoremstyle{definition}

\theoremstyle{remark}
\newtheorem{remark}[theorem]{Remark}

\newcommand{\PP}{\mathbb P}
\newcommand{\QQ}{\mathbb Q}
\newcommand{\ZZ}{\mathbb Z}
\newcommand{\CC}{\mathbb C}
\newcommand{\OO}{\mathcal O}
\newcommand{\Li}{\operatorname{Li}}
\newcommand{\Gr}{\operatorname{Gr}}
\newcommand{\Sym}{\operatorname{Sym}}

\title{Mirror symmetry for gCICY threefolds (II):\\ $[3,-1]$~blocks, VGIT, and toric phases}
\author{Atsushi Kanazawa}
\date{}

\begin{document}
\maketitle

\begin{abstract}
We study mirror symmetry for a class of generalized complete intersection Calabi--Yau (gCICY) threefolds via variations of GIT. 
Our focus is on gCICY threefolds in $\PP^1\times B$ with a $[3,-1]$~block, cut out by sections of $\OO(3)\boxtimes\OO_B(L)$ and $\OO(-1)\boxtimes\OO_B(L+R)$, where $B$ is a smooth projective fourfold and $2L+R=-K_B$. 
Auxiliary variables turn the generalized section into a relative variation of GIT, whose two phases are the gCICY threefold itself and the zero locus of a rank-3 bundle on $\PP^2\times B$, and they are isomorphic or related by $\int_B L^4$ disjoint Atiyah flops.

Over products of projective spaces, there are 18 Calabi--Yau threefold configurations with a $[3,-1]$~block.
All of them except the non-Gorenstein case studied in the companion paper have toric phases with Gorenstein ambient spaces and nef partitions, so that after the variation of GIT their mirror symmetry reduces to Batyrev--Borisov duality.
We analyze four examples with a $[3,-1]$~block and one outside this class in detail.
\end{abstract}

\tableofcontents

%%%%%%%%%%%%%%%%%%%%%%%%%%%%%%%%%%%%%%%%%%%%%%%%%%%%%%%%%%%%%%%%%%%%%%%%%%%%%%%%%%%%%%%%%%%%%%
\section{Introduction}\label{sec:intro}

The generalized complete intersection Calabi--Yau (gCICY) construction, introduced by Anderson, Apruzzi, Gao, Gray and Lee \cite{AAGGL}, 
extends the conventional CICY construction by allowing the later defining line bundles to have negative degrees in the ambient product of projective spaces.
A negative degree stands for a section that exists only after restriction: we first cut out an intermediate variety, on which a line bundle with no sections on the ambient space may acquire sections.
Such a section, which in general does not extend to the ambient space, is called a \emph{generalized section} \cite{AAGGL}.
Garbagnati and van Geemen made this mechanism explicit \cite{GG}.
The cohomological, fibration, physical and computational aspects of gCICYs have since been developed in \cite{BH,JiaLin,LPS,CuiGaoWang}.

A basic question is whether a generalized presentation describes a new kind of geometry or hides a more conventional birational or toric model, and the question bears directly on mirror symmetry.
The Batyrev--Borisov construction applies to nef partitions of Gorenstein toric Fano varieties \cite{BB}, and the Hori--Vafa equations \cite{HV} give the corresponding torus mirrors. 
Since a defining line bundle with a negative degree is not nef on the ambient space, neither construction applies directly.
Berglund and H\"ubsch studied mirror symmetry for gCICYs in Hirzebruch $n$-folds, where the generalized sections appear as Laurent polynomials, 
and proposed mirrors by extending Batyrev's construction to non-Fano toric ambient spaces and non-convex polytopes \cite{BH,BHmirror,BHTyurin}. 
Our approach, initiated in \cite{KanCompanion}, uses instead auxiliary variables and variation of GIT to find a birational model in which the negative entry disappears, and then apply the Batyrev--Borisov or Hori--Vafa construction there.
We call a birational model admitting a nef complete intersection or anticanonical hypersurface presentation in a toric ambient variety a \emph{toric phase}. 

Our starting point is the gCICY threefold
\[
X_{0,+}=\left[\begin{array}{c|cc}
\PP^1&3&-1\\
\PP^4&2&3
\end{array}\right], \qquad (h^{1,1},h^{2,1})=(2,46),
\]
studied in the companion paper \cite{KanCompanion}.
An auxiliary presentation shows that $X_{0,+}$ is related by a flop of 16 disjoint $(-1,-1)$-curves to a toric phase $X_{0,-}$, a complete intersection in a $\QQ$-Fano toric eightfold.
However, because that eightfold fails to be Gorenstein, the Batyrev--Borisov construction is not available.
Instead, a mirror family is selected by the Hori--Vafa equations and given a smooth projective crepant model $Y_0$. 
Its periods have two maximally unipotent boundary points, whose indicial algebras are isomorphic to the even cohomology rings $H^{\mathrm{even}}(X_{0,+},\CC)$ and $H^{\mathrm{even}}(X_{0,-},\CC)$, respectively.   
The basic geometric picture is
\[
\xymatrix{
X_{0,+}   \ar@{-->}[rr]^{\text{flop}}  \ar@{<->}[rd]_{\text{mirror}}  &  &  X_{0,-}  \ar@{<->}[ld]^{\text{mirror}}  \\
 &  Y_0 &
}
\]

The present paper clarifies the mechanism behind this case and determines how far it extends.
A \emph{$[3,-1]$~block} means a pair of columns whose degrees along a distinguished $\PP^1$-factor are $3$ and $-1$.
The Calabi--Yau threefold configurations considered in this paper have the form
\[
X_+=
\left[\begin{array}{c|cc}
\PP^1&3&-1\\
B&L&L+R
\end{array}\right],\qquad 2L+R=-K_B,
\]
where $B$ is a smooth projective fourfold and $L$ and $R$ are divisor classes on $B$.
More generally, a collection of $k$ columns whose degrees along a distinguished $\PP^1$-factor are $d_1,\ldots,d_k$, with $d_1\ge\cdots\ge d_{k-1}>0>d_k$, is called a
\emph{$[d_1,\ldots,d_k]$~block}.  For the Calabi--Yau threefold
configurations considered here, $\sum_{i=1}^k d_i=2$ and $\dim B=k+2$.\footnote{
The forthcoming paper \cite{KanIII} studies gCICYs with a $[4,-2]$~block, for which the syzygy module and the base scheme of the generalized anticanonical system become essential and the opposite phases are a Grassmannian phase and a toric-join phase.} 
Section~\ref{sec:beyond} contains configurations with a $[2,1,-1]$~block that reduce to ordinary CICYs. 
For a $[3,-1]$~block, the cubic evaluation sequence provides an auxiliary presentation of the generalized section, which induces a relative variation of GIT for a $\CC^*$-action over $B$. 
This variation of GIT has two phases: one is the gCICY $X_+$, which we call the \emph{generalized phase}, and the other, the \emph{opposite phase}, is the zero locus $X_-$ of a rank-3 bundle on $\PP^2\times B$.
If $\OO_B(L)$ and $\OO_B(R)$ are globally generated and the sections are general, then $X_+$ and $X_-$ are either isomorphic or related by the simultaneous Atiyah flop of
\[
N_{\mathrm{flop}}=\int_B L^4
\]
disjoint $(-1,-1)$-curves. 
Over products of projective spaces, there are 18 Calabi--Yau threefold configurations with a $[3,-1]$~block.
Among these, 7 involve flops and $N_{\mathrm{flop}}$ is either $1,4,6,12,16,24,32$ (Proposition~\ref{prop:classification}).
When $R=0$, both phases are ordinary CICYs, namely the two small resolutions of a determinantal hypersurface (Proposition~\ref{prop:Rzero}).

For the opposite phase $X_-$, we establish two kinds of results.
Firstly, suppose $B$ is a product of projective spaces $\PP^{n_1}\times\cdots\times\PP^{n_k}$ with hyperplane classes $H_i$, and write $L=\sum_{i=1}^k a_iH_i$ and $R=\sum_{i=1}^k r_iH_i$.
The ambient space of $X_-$ is Gorenstein and the 5 defining classes form a Cartier nef partition, for a suitable chamber, if $a_i\mid r_i$ whenever $a_i\ge2$ (Lemma~\ref{lem:stabilizers} and Proposition~\ref{prop:BBduality}).
In this case, Batyrev--Borisov duality applies to $X_-$.
Since $X_+$ and $X_-$ are isomorphic or related by flops, they have the same Hodge numbers, and hence the Batyrev--Borisov mirror $Y$ of $X_-$ satisfies $h^{p,q}(Y)=h^{3-p,q}(X_+)$ (Theorem~\ref{thm:BBmirror}).
Among the 18 configurations, the condition fails only for the companion-paper example, whose opposite phase lies in a toric ambient space with non-Gorenstein $\frac12(1,1,1,1,1)$ singularities \cite{KanCompanion}. 
Secondly, we compute the cubic forms, second Chern classes, and Euler characteristics of both phases over a general base $B$, including their changes under flops.
We also determine the Picard groups under additional positivity assumptions.
For products of projective spaces, we study the natural K3 and genus-one fibrations and derive a formula for the fundamental period of $X_-$ (Sections~\ref{sec:uniform} and~\ref{sec:fibperiod}).

We study the following four examples in detail:
\begin{align}
X_{1,+}&=\left[\begin{array}{c|cc}
\PP^1&3&-1\\ \PP^3&2&2\\ \PP^1&1&1
\end{array}\right],
& (h^{1,1},h^{2,1})&=(3,55),\label{eq:X1intro}\\
X_{2,+}&=\left[\begin{array}{c|cc}
\PP^1&3&-1\\ \PP^3&2&2\\ \PP^1&0&2
\end{array}\right],
& (h^{1,1},h^{2,1})&=(3,31),\label{eq:X2intro}\\
X_{3,+}&=\left[\begin{array}{c|cc}
\PP^1&3&-1\\ \PP^4&1&4
\end{array}\right],
& (h^{1,1},h^{2,1})&=(2,86),\label{eq:X3intro}\\
X_{4,+}&=\left[\begin{array}{c|cc}
\PP^1&3&-1\\ \PP^2&1&2\\ \PP^2&1&2
\end{array}\right],
& (h^{1,1},h^{2,1})&=(3,42).\label{eq:X4intro}
\end{align}
These examples show different features: in the first, $R=0$, and both phases are ordinary CICYs of the same type, related by a $32$-curve flop.
In the second, the transition is an isomorphism.
The third reaches an anticanonical hypersurface in a smooth toric Fano fourfold after a single flop.
The fourth is symmetric in the two factors of its base and has a 6-curve flop to a toric phase with 3 genus-one fibrations, of indices 3, 1 and 1. 

The auxiliary-variable mechanism is not restricted to $[3,-1]$~blocks, and we treat configurations outside this class in Section~\ref{sec:beyond}. 
When a negative column is paired with a column that is linear in a $\PP^1$ factor, the generalized section can be eliminated altogether, and the configuration is an ordinary CICY in disguise (Lemma~\ref{lem:disguise}).
For instance,
\[
\left[\begin{array}{c|ccc}
\PP^1&1&-1&2\\ \PP^5&1&1&4
\end{array}\right]
\cong
\left[\begin{array}{c|c}
\PP^1&2\\ \PP^3&4
\end{array}\right], \qquad
\left[\begin{array}{c|ccc}
\PP^1&1&-1&2\\ \PP^1&1&1&0\\ \PP^4&1&1&3
\end{array}\right]
\cong
\left[\begin{array}{c|ccc}
\PP^1&0&0&2\\ \PP^1&1&1&0\\ \PP^4&1&1&3
\end{array}\right]. 
\]
When two negative columns act on two different $\PP^1$ factors, the same elimination produces a determinantal phase which is not a complete intersection.
Our fifth example,
\begin{equation}
X_{5,+}=\left[\begin{array}{c|cccc}
\PP^1&1&1&-1&1\\
\PP^1&1&1&1&-1\\
\PP^5&3&1&1&1
\end{array}\right], \qquad (h^{1,1},h^{2,1})=(3,81),
\end{equation}
is of this kind. 
It reduces, without a flop, to an anticanonical hypersurface in a smooth toric Fano fourfold.

Examples 1, 2, 3 and 5 appear in \cite{AAGGL}, but Example 4 does not.
We derive the geometric descriptions and the Hodge numbers from the GIT presentations, except $h^{1,1}(X_{2,+})=3$, for which we use \cite{AAGGL}.
Table~\ref{tab:mainexamples} summarizes the comparison.

\begin{table}[ht]
\centering
\small
\begin{tabular}{c|c|c|c}
phase & $(h^{1,1},h^{2,1})$ & transition & toric phase \\
\midrule
$X_{1,+}$ & $(3,55)$ & $32$-curve flop & ordinary CICY (both) \\
$X_{2,+}$ & $(3,31)$ & isomorphism & Gorenstein toric CI \\
$X_{3,+}$ & $(2,86)$ & $1$-curve flop & toric Fano hypersurface \\
$X_{4,+}$ & $(3,42)$ & $6$-curve flop & smooth toric CI \\
\midrule
$X_{5,+}$ & $(3,81)$ & direct reduction & toric Fano hypersurface \\
\end{tabular}
\caption{Four examples with a $[3,-1]$ block and one outside this class.}
\label{tab:mainexamples}
\end{table} 

For each example, we construct a mirror family and compute its fundamental period and Picard--Fuchs system. 
The indicial algebra at the maximally unipotent boundary point associated with the toric phase is isomorphic to its complex even cohomology ring.

\paragraph{Conventions.}
We work over $\CC$. 
A Calabi--Yau threefold $X$ is a smooth projective threefold with $K_X=0$ and $H^1(X,\OO_X)=H^2(X,\OO_X)=0$. 
``General'' means on a dense Zariski-open subset of the relevant parameter space.

%%%%%%%%%%%%%%%%%%%%%%%%%%%%%%%%%%%%%%%%%%%%%%%%%%%%%%%%%%%%%%%%%%%%%%%%%%%%%%%%%%%%%%%%%%%%%%
\section{\texorpdfstring{gCICYs with a $[3,-1]$~block and the opposite GIT phase}{gCICYs with a [3,-1] block and the opposite GIT phase}}
\label{sec:cubicblock}

\subsection{The auxiliary construction}

\paragraph{Setting.}
Let $B$ be a smooth projective fourfold and let $L,R$ be divisor classes on $B$ with
\begin{equation}
2L+R=-K_B.
\label{eq:CYcondition}
\end{equation}
When $B$ is a product of projective spaces, $X_+$ below is a gCICY in the sense of \cite{AAGGL}. 
For general $B$ we keep the same configuration notation, with the row of $B$ recording divisor classes on $B$.
We write $F$ for the pullback to $\PP^1\times B$ of the hyperplane class of $\PP^1_{[u:v]}$ and use the same letter for its restrictions.
We take
\begin{equation}
P(u,v):=f_0u^3+f_1u^2v+f_2uv^2+f_3v^3\in H^0\bigl(\PP^1\times B,\OO(3)\boxtimes\OO_B(L)\bigr), \qquad f_i\in H^0(B,\OO_B(L)),
\label{eq:cubicP}
\end{equation}
and write $M:=Z(P)\subset\PP^1\times B$.
The generalized phase is the zero locus
\[
X_+:=Z_M(s),\qquad s\in H^0\bigl(M,(\OO_{\PP^1}(-1)\boxtimes\OO_B(L+R))|_M\bigr),
\]
which is the gCICY with configuration
$\left[\begin{smallmatrix}\PP^1&3&-1\\ B&L&L+R\end{smallmatrix}\right]$.
The two defining classes are $3F+L$ and $-F+L+R$, with classes from $B$ pulled back to $\PP^1\times B$.
Since $K_{\PP^1\times B}=-2F+K_B$, condition \eqref{eq:CYcondition} is equivalent to $K_{X_+}=0$.
The section $s$ is the
\emph{generalized section} of the block as $\OO_{\PP^1}(-1)\boxtimes\OO_B(L+R)$
has no nonzero sections on $\PP^1\times B$, and $s$ exists only after restriction to $M$.

\paragraph{The syzygy of the cubic.}
The cubic evaluation sequence on $\PP^1$ is
\begin{equation}
0\longrightarrow\OO(-1)^{\oplus3}
\xrightarrow{D(u,v)}\OO^{\oplus4}
\xrightarrow{(u^3,u^2v,uv^2,v^3)}\OO(3)
\longrightarrow0, \qquad
D(u,v):=\begin{pmatrix}
-v&0&0\\ u&-v&0\\ 0&u&-v\\ 0&0&u
\end{pmatrix}.
\label{eq:cubicsyzygy}
\end{equation}
It is exact because $u^3,v^3$ have no common zero and $D(u,v)$ always has rank 3.
Tensoring by the pullback of $\OO_B(L)$ and restricting to $M$, the condition $P=0$ says that $\mathbf f=(f_0,f_1,f_2,f_3)^t$ lies in the image of $D(u,v)$.
Therefore on $M$ there is a unique vector $y=(y_1,y_2,y_3)^t$ of sections of $\OO_{\PP^1}(-1)\boxtimes\OO_B(L)$ with
\begin{equation}
D(u,v)y=\mathbf f.
\label{eq:auxgeneral}
\end{equation}
Since $D(u,v)$ is injective on every fiber, the zero locus of $y$ is $\PP^1\times Z_B(f_0,\ldots,f_3)$.

\begin{lemma}\label{lem:allsections}
Assume that $f_0,\ldots,f_3$ are linearly independent and that $Z_B(f_0)$ is a reduced divisor; both hold for general $f_i$ if $h^0(B,\OO_B(L))\ge4$ and $\OO_B(L)$ is globally generated.
Then $y_1,y_2,y_3$ form a basis of $H^0(M,(\OO_{\PP^1}(-1)\boxtimes\OO_B(L))|_M)$, and every generalized section has the form
\begin{equation}
s=\ell_1y_1+\ell_2y_2+\ell_3y_3, \qquad \ell_j\in H^0(B,\OO_B(R)).
\end{equation}
\end{lemma}

\begin{proof}
For a divisor class $N$ on $B$, the Koszul sequence of $M$ twisted by $\OO(-1)\boxtimes\OO_B(N)$ is
\[
0\to\OO(-4)\boxtimes\OO_B(N-L)\xrightarrow{\ \times P\ }\OO(-1)\boxtimes\OO_B(N)\to
(\OO(-1)\boxtimes\OO_B(N))|_M\to0 .
\]
The middle term has no cohomology by the K\"unneth formula, since $\OO_{\PP^1}(-1)$ is acyclic.
Hence the connecting map
\[
\begin{aligned}
\delta_N\colon H^0\bigl(M,(\OO(-1)\boxtimes\OO_B(N))|_M\bigr)
&\xrightarrow{\ \sim\ }H^1\bigl(\OO(-4)\boxtimes\OO_B(N-L)\bigr)\\
&=H^1(\PP^1,\OO(-4))\otimes H^0(B,\OO_B(N-L))
\end{aligned}
\]
is an isomorphism.
For $N=L$ the target is 3-dimensional, and for $N=L+R$ it is $H^1(\PP^1,\OO(-4))\otimes H^0(B,\OO_B(R))$.
Multiplication by $\ell\in H^0(B,\OO_B(R))$ maps the sequence for $N=L$ to the sequence for $N=L+R$, so $\delta_{L+R}(\ell\sigma)=\ell\,\delta_L(\sigma)$.
Therefore, multiplication induces an isomorphism $H^0(M,(\OO(-1)\boxtimes\OO_B(L))|_M)\otimes H^0(B,\OO_B(R))\to H^0(M,(\OO(-1)\boxtimes\OO_B(L+R))|_M)$.

It suffices to check that the $y_j$ are linearly independent.
On $M\cap\{v=0\}=\{[1:0]\}\times Z_B(f_0)$, equation \eqref{eq:auxgeneral} reads $y_j=f_j$.
If $\sum_{j=1}^3c_jy_j=0$, then $\sum_{j=1}^3c_jf_j$ vanishes on the reduced divisor $Z_B(f_0)$ of class $L$, hence is a constant multiple of $f_0$.
By linear independence of the $f_i$, all $c_j$ vanish.
\end{proof}

The polynomial system
\begin{equation}
D(u,v)y=\mathbf f,\qquad \ell\cdot y=0
\label{eq:auxsystemII}
\end{equation}
is the \emph{auxiliary presentation} of the generalized phase.
Note that the negative degree has moved into the auxiliary variables: $y_j$ has degree $-1$ in the $\PP^1$ direction and degree $L$ on $B$.

\paragraph{The relative variation of GIT.}
Let $\mathbb V\to B$ be the total space of the vector bundle $\OO_B^{\oplus2}\oplus\OO_B(L)^{\oplus3}$, with fiber coordinates $(u,v)$ and $y=(y_1,y_2,y_3)$, and let $\CC^*$ act on the fibers with weights $+1$ on $(u,v)$ and $-1$ on $y$.
The equations $D(u,v)y=\mathbf f$ are invariant sections of the pullback of $\OO_B(L)$, and $\ell\cdot y$ is a section of the pullback of $\OO_B(L+R)$ of weight $-1$.
Let $Z\subset\mathbb V$ be their common zero locus.
We regard $Z$, together with the fiberwise $\CC^*$-action, as the \emph{relative VGIT master space} for the two Calabi--Yau phases.
The two nontrivial linearizations, by the characters $\pm1$, have semistable loci
\[
\mathbb V^{+}:=\{(u,v)\ne0\},\qquad \mathbb V^{-}:=\{y\ne0\},
\]
on which $\CC^*$ acts freely.
The quotients are
\[
\begin{aligned}
\mathbb V^+/\CC^*&=\operatorname{Tot}\bigl((\OO_{\PP^1}(-1)\boxtimes\OO_B(L))^{\oplus3}\bigr)\to\PP^1\times B,\\
\mathbb V^-/\CC^*&=\operatorname{Tot}\bigl((\OO_{\PP^2}(-1)\boxtimes\OO_B(L))^{\oplus2}\bigr)\to\PP^2_y\times B,
\end{aligned}
\]
where we use $\PP(\OO_B(L)^{\oplus3})\cong\PP^2\times B$ for the second.
Fiberwise over $B$ this is the standard flip between $\operatorname{Tot}(\OO_{\PP^1}(-1)^{\oplus3})$ and $\operatorname{Tot}(\OO_{\PP^2}(-1)^{\oplus2})$.
In the plus quotient, the first 4 equations of \eqref{eq:auxsystemII} cut out the graph of the section $y$ over $M$, and the last equation cuts out $X_+$:
\[
Z^+/\CC^*\cong X_+ ,\qquad Z^\pm:=Z\cap\mathbb V^\pm .
\]
\begin{remark}
When $B=\PP^4$, $L=2H$ and $R=H$, this is the toric VGIT studied in \cite[Section~2.2]{KanCompanion}. 
\end{remark}

\paragraph{The minus phase.}
To describe $Z^-/\CC^*$, write
\[
D(u,v)y=A(y)\binom uv, \qquad
A(y):=\begin{pmatrix}
0&-y_1\\ y_1&-y_2\\ y_2&-y_3\\ y_3&0
\end{pmatrix}.
\]
Note that $A(y)$ has rank 2 for every $[y]\in\PP^2$.
Define the rank-2 bundle $\mathcal Q$ on $\PP^2$ by
\begin{equation}
0\longrightarrow\OO_{\PP^2}(-1)^{\oplus2}
\xrightarrow{A(y)}\OO_{\PP^2}^{\oplus4}
\longrightarrow\mathcal Q\longrightarrow0.
\label{eq:Qgeneral}
\end{equation}
For fixed $(y,x)$, we can solve the equation $A(y)(u,v)^t=\mathbf f(x)$ if and only if the image of $\mathbf f(x)$ in $\mathcal Q_y$ vanishes, and then the solution is unique.
Therefore $Z^-/\CC^*$ is the zero locus
\begin{equation}
X_-:= Z\!\left((\mathcal Q\boxtimes\OO_B(L))\oplus(\OO_{\PP^2}(1)\boxtimes\OO_B(R))\right)
\subset\PP^2_y\times B.
\label{eq:generalminus}
\end{equation}
Since $c_1(\mathcal Q)=2G$, where $G=c_1(\OO_{\PP^2}(1))$, the bundle in \eqref{eq:generalminus} has determinant class
\[
3G+2L+R=-K_{\PP^2\times B}.
\]
Therefore $K_{X_-}=0$, if $X_-$ is smooth of dimension 3.
On the common open set $(Z\cap\mathbb V^+\cap\mathbb V^-)/\CC^*$ of the two phases, the coordinates $y_j$ are sections of $\OO_{\PP^2}(1)$ on the minus side and of $\OO_{\PP^1}(-1)\boxtimes\OO_B(L)$ on the plus side, so
\begin{equation}
\label{eq:Gclass}
G=L-F .
\end{equation}
It is the change of basis between the generalized and opposite phases in every example below.

%%%%%%%%%%%%%%%%%%%%%%%%%%%%%%%%%%%%%%%%%%%%%%%%%%%%%%%%%%%%%%%%%%%%%%%%%%%%%%%%%%%%%%%%%%%%%%

\subsection{The transition theorem}

\begin{lemma}\label{lem:cubicsmooth}
Assume that $\OO_B(L)$ and $\OO_B(R)$ are globally generated.
Let $f_i\in H^0(B,\OO_B(L))$ and $\ell_j\in H^0(B,\OO_B(R))$ be general.
Then $M$, $X_+$ and $X_-$ are smooth of the expected dimensions and $K_{X_\pm}=0$.
The common zero locus $Z_B(f_0,\ldots,f_3)$ is empty or a reduced set of $\int_B L^4$ points, where the $\ell_j$ have no common zero. 
\end{lemma}

\begin{proof}
The sections of $\OO_{\PP^1}(3)\boxtimes\OO_B(L)$ of the form \eqref{eq:cubicP}, and the sections of $(\mathcal Q\boxtimes\OO_B(L))\oplus(\OO_{\PP^2}(1)\boxtimes\OO_B(R))$ induced by the data $(\mathbf f,\ell)$ as in \eqref{eq:generalminus}, 
generate these bundles, because $\OO_B(L)$ and $\OO_B(R)$ are globally generated and $\mathcal Q$ is a quotient of $\OO^{\oplus4}$. 
Therefore $M$ and $X_-$ are smooth of the expected dimension by Bertini's theorem for a generating linear system.
General sections $f_0,\ldots,f_3$ of the globally generated line bundle $\OO_B(L)$ on the fourfold $B$ meet in a reduced set of $\int_B L^4$ points (possibly none), where general $\ell_j$ do not vanish simultaneously. 
At a point of $M$ such that $y_j\ne0$, choosing $\ell_j$ nonzero at the image in $B$ and the other $\ell$'s zero shows that the linear system $\{\ell\cdot y\}$ is base-point free on $M\setminus(\PP^1\times Z_B(f))$, 
and hence, by Bertini's theorem, $X_+$ is smooth there.
Along a curve $C=\PP^1\times\{x\}$ such that $x\in Z_B(f)$, transversality identifies $N_{C/M}$ with the kernel of $\OO_{\PP^1}^{\oplus4}\to\OO_{\PP^1}(3)$, 
$(\xi_i)\mapsto\sum_{i=0}^3u^{3-i}v^i\xi_i$, which is $\OO(-1)^{\oplus3}$ by \eqref{eq:cubicsyzygy}. 
Differentiating \eqref{eq:auxgeneral} shows that $dy\colon N_{C/M}\to(\OO(-1)\boxtimes\OO_B(L))^{\oplus3}|_C$ is an isomorphism.
The normal derivative of $\ell\cdot y$ along $C$ is therefore the constant map $\OO(-1)^{\oplus3}\to\OO(-1)$ given by $\ell(x)\ne0$, which is surjective. 
Therefore $X_+$ is smooth along $C$ and $N_{C/X_+}\cong\OO(-1)^{\oplus2}$.
The relation $K_{X_\pm}=0$ follows from \eqref{eq:CYcondition} and the adjunction formula. 
\end{proof}

\begin{theorem}
\label{thm:cubicblock}
Under the assumptions of Lemma~\ref{lem:cubicsmooth}, we set $N_{\mathrm{flop}}:=\int_B L^4$.
\begin{enumerate}
\item $X_+\cong Z^+/\CC^*$ and $X_-\cong Z^-/\CC^*$ are isomorphic away
from the curves $C_x=\PP^1\times\{x\}\subset X_+$ and the lines $C'_x=\{\ell(x)\cdot y=0\}\times\{x\}\subset X_-$, $x\in Z_B(f)$.
\item If $N_{\mathrm{flop}}=0$, then $X_+\cong X_-$. If $N_{\mathrm{flop}}>0$, then $X_-$ is obtained from $X_+$ by the simultaneous Atiyah flop of the $N_{\mathrm{flop}}$ disjoint $(-1,-1)$-curves $C_x$, and $F\cdot C_x=1$, $G\cdot C'_x=1$.
\item If $B=\PP^{n_1}\times\cdots\times\PP^{n_k}$ and $L\ne0$, 
then $h^i(\OO_{X_\pm})=1,0,0,1$ for $i=0,1,2,3$, and in particular $X_\pm$ are Calabi--Yau threefolds.
\end{enumerate}
\end{theorem}

\begin{proof}
For assertion 1, by construction $X_\pm=Z^\pm/\CC^*$, and both contain the quotient of $Z\cap\mathbb V^+\cap\mathbb V^-$.
A point of $Z^+$ such that $y=0$ satisfies $\mathbf f(x)=D(u,v)y=0$, so it lies on some $C_x$. 
A point of $Z^-$ such that $(u,v)=0$ satisfies $\mathbf f(x)=A(y)(u,v)^t=0$ and $\ell(x)\cdot y=0$, so it lies on some $C'_x$.
Conversely these curves lie in the respective phases.

Let us next show assertion 2. 
If $N_{\mathrm{flop}}=0$, then $Z_B(f)=\emptyset$ by Lemma~\ref{lem:cubicsmooth}, and $X_+\cong X_-$ by assertion 1. 
Now suppose $N_{\mathrm{flop}}>0$, and consider the projections $\pi_\pm\colon X_\pm\to B$.
Over $x\notin Z_B(f)$ the fibers are finite. 
On $X_+$ the cubic $P(u,v,x)$ is not zero, and on $X_-$ the class of $\mathbf f(x)$ in $\mathcal Q_y$ vanishes only at finitely many $[y]$.
Over $x\in Z_B(f)$ the fibers are $C_x$ and $C'_x$.
Let $X_\pm\to\overline X_\pm\to B$ be the Stein factorizations.
The maps $X_\pm\to\overline X_\pm$ contract exactly the curves $C_x$ (resp. $C'_x$).
Since $X_+$ and $X_-$ are normal and isomorphic outside subsets of codimension 2, the sheaves $\pi_{\pm*}\OO_{X_\pm}$ coincide, so $\overline X_+=\overline X_-=:\overline X$.
The divisor $F$ satisfies $F\cdot C_x=1$, so it is ample relative to $X_+\to\overline X$.
On $X_-$, $\OO_{X_-}(G)$ restricts to $\OO_{C'_x}(1)$ on the line $C'_x\subset\PP^2_y$, while $\OO_{X_-}(L)$ restricts to the trivial bundle, so by \eqref{eq:Gclass} the transform of $F$ satisfies $F\cdot C'_x=L\cdot C'_x-G\cdot C'_x=-1$.
Therefore $X_-\to\overline X$ is the flop of $X_+\to\overline X$ with respect to $F$, which is unique \cite[Section~6.1]{KM}.
Since $N_{C_x/X_+}\cong\OO(-1)^{\oplus2}$ by Lemma~\ref{lem:cubicsmooth}, it is the simultaneous Atiyah flop of the $N_{\mathrm{flop}}$ curves $C_x$.

For assertion 3, we choose a factor $\PP^{n_i}$ on which $L$ has positive degree $a_i$, and let $r_i$ be the degree of $R$ there.
By \eqref{eq:CYcondition}, $2a_i+r_i=n_i+1$ with $r_i\ge0$, so $1\le a_i\le n_i$ and $1\le a_i+r_i\le n_i$.
Hence $H^*(\PP^{n_i},\OO(-a_i))$ and $H^*(\PP^{n_i},\OO(-a_i-r_i))$ vanish, and by K\"unneth so do $H^*(B,\OO_B(-L))$ and $H^*(B,\OO_B(-L-R))$.
In the Koszul complex of \eqref{eq:generalminus} the terms other than $\OO$ and $\OO(K_{\PP^2\times B})$ are $\mathcal Q^*\boxtimes\OO_B(-L)$, $\OO_{\PP^2}(-1)\boxtimes\OO_B(-R)$, $\OO_{\PP^2}(-2)\boxtimes\OO_B(-2L)$ and $\mathcal Q^*(-1)\boxtimes\OO_B(-L-R)$.
The first and fourth are acyclic by the vanishing just shown, and the second and third because $\OO_{\PP^2}(-1)$ and $\OO_{\PP^2}(-2)$ are acyclic.
Hence $h^i(\OO_{X_-})=1,0,0,1$.
Similarly for $X_+$, since the $h^{i,0}$ are birational invariants. 
\end{proof}

In particular, the identification \eqref{eq:Gclass} extends to $\operatorname{Pic}(X_\pm)$, since the complements of the common open set in $X_\pm$ have codimension 2 by assertion 1.

%%%%%%%%%%%%%%%%%%%%%%%%%%%%%%%%%%%%%%%%%%%%%%%%%%%%%%%%%%%%%%%%%%%%%%%%%%%%%%%%%%%%%%%%%%%%%%

\subsection{Classification over products of projective spaces}

Let $B=\PP^{n_1}\times\cdots\times\PP^{n_k}$ such that $\sum_{i=1}^kn_i=4$ and $H_i$ be the hyperplane classes. 
We write $L=\sum_{i=1}^ka_iH_i$ and $R=\sum_{i=1}^kr_iH_i$.
Condition \eqref{eq:CYcondition} reads
\begin{equation}\label{eq:factorCY}
2a_i+r_i=n_i+1,\qquad a_i,r_i\ge0,
\end{equation}
factor by factor, and we require $L\ne0$, so that Theorem~\ref{thm:cubicblock}(3) applies.
The factor $\PP^{n_i}$ contributes the row $(a_i,\,a_i+r_i)$ to the configuration matrix, below the row $[3,-1]$ of the cubic $\PP^1$.

\begin{proposition}\label{prop:classification}
Up to permutation of the factors there are 18 Calabi--Yau threefold configurations with a $[3,-1]$~block over products of projective spaces; they are listed in Table~\ref{tab:flopcounts}.
Moreover, the number $N_{\mathrm{flop}}$ of flopping curves is given by
\[
\int_B L^4=
\begin{cases}
4!\prod_{i=1}^k\dfrac{a_i^{n_i}}{n_i!}, & a_i\ge1\text{ for all }i,\\[1ex]
0, & \text{otherwise},
\end{cases}
\]
7 configurations have $N_{\mathrm{flop}}>0$, and $N_{\mathrm{flop}}$ is either $1,4,6,12,16,24,32$. 
For the other 11 configurations, the transition is an isomorphism. 
\end{proposition}

\begin{proof}
By \eqref{eq:factorCY} the possible pairs $(a_i,r_i)$ are
\[
\PP^1\colon(0,2),(1,0);\qquad
\PP^2\colon(0,3),(1,1);\qquad
\PP^3\colon(0,4),(1,2),(2,0);\qquad
\PP^4\colon(0,5),(1,3),(2,1).
\]
Running over the partitions $4$, $3+1$, $2+2$, $2+1+1$ and $1+1+1+1$ of $\dim B$, choosing one pair for each factor, excluding $L=0$, and identifying choices that differ by a permutation of equal factors, we obtain the 18 rows of the table.
Since $H_i^{n_i+1}=0$, the multinomial expansion of $(\sum_{i=1}^ka_iH_i)^4$ has a single nonzero term, $4!\prod_{i=1}^k(a_iH_i)^{n_i}/n_i!$, and it is nonzero if and only if all $a_i\ge1$.
The list of values follows from the table.
\end{proof}

\begin{table}[ht]
\centering
\footnotesize
\setlength{\tabcolsep}{4pt}
\begin{tabular}{c|c|c|c|c|c|c}
$B$ & $L$ & $R$ & rows over $B$ & $N_{\mathrm{flop}}$ & $e$ & notes \\ \hline
$\PP^4$ & $H$ & $3H$ & $(1,4)$ & 1 & $-168$ &  Ex.~3\\
$\PP^4$ & $2H$ & $H$ & $(2,3)$ & 16 & $-88$ &  non-Gor. \cite{KanCompanion}\\
$\PP^3\times \PP^1$ & $T$ & $4H$ & $(0,4;\,1,1)$ & 0 & $-168$ &  $^\dagger$\\
$\PP^3\times \PP^1$ & $H$ & $2H+2T$ & $(1,3;\,0,2)$ & 0 & $-144$ &  \\
$\PP^3\times \PP^1$ & $H+T$ & $2H$ & $(1,3;\,1,1)$ & 4 & $-72$ &  \\
$\PP^3\times \PP^1$ & $2H$ & $2T$ & $(2,2;\,0,2)$ & 0 & $-56$ &  Ex.~2\\
$\PP^3\times \PP^1$ & $2H+T$ & $0$ & $(2,2;\,1,1)$ & 32 & $-104$ &  Ex.~1\\
$(\PP^2)^{2}$ & $H_2$ & $3H_1+H_2$ & $(0,3;\,1,2)$ & 0 & $-144$ &  $^\dagger$\\
$(\PP^2)^{2}$ & $H_1+H_2$ & $H_1+H_2$ & $(1,2;\,1,2)$ & 6 & $-78$ &  Ex.~4\\
$\PP^2\times (\PP^1)^{2}$ & $T_2$ & $3H+2T_1$ & $(0,3;\,0,2;\,1,1)$ & 0 & $-144$ & $^\dagger$\\
$\PP^2\times (\PP^1)^{2}$ & $T_1+T_2$ & $3H$ & $(0,3;\,1,1;\,1,1)$ & 0 & $-36$ &  \\
$\PP^2\times (\PP^1)^{2}$ & $H$ & $H+2T_1+2T_2$ & $(1,2;\,0,2;\,0,2)$ & 0 & $-128$ & $^\dagger$\\
$\PP^2\times (\PP^1)^{2}$ & $H+T_2$ & $H+2T_1$ & $(1,2;\,0,2;\,1,1)$ & 0 & $-56$ &  \\
$\PP^2\times (\PP^1)^{2}$ & $H+T_1+T_2$ & $H$ & $(1,2;\,1,1;\,1,1)$ & 12 & $-68$ &  \\
$(\PP^1)^{4}$ & $T_4$ & $2T_1+2T_2+2T_3$ & $(0,2;\,0,2;\,0,2;\,1,1)$ & 0 & $-128$ &  $^\dagger$\\
$(\PP^1)^{4}$ & $T_3+T_4$ & $2T_1+2T_2$ & $(0,2;\,0,2;\,1,1;\,1,1)$ & 0 & $-32$ & \\
$(\PP^1)^{4}$ & $T_2+T_3+T_4$ & $2T_1$ & $(0,2;\,1,1;\,1,1;\,1,1)$ & 0 & $-48$ & \\
$(\PP^1)^{4}$ & $T_1+T_2+T_3+T_4$ & $0$ & $(1,1;\,1,1;\,1,1;\,1,1)$ & 24 & $-80$ & \\
\end{tabular}
\caption{Calabi--Yau threefold configurations with a $[3,-1]$~block
over products of projective spaces. 
Here $H$, $H_i$ and $T$, $T_i$ are the hyperplane classes of the factors $\PP^n$ with $n\ge2$ and $\PP^1$. 
In the rows marked $^\dagger$, we have $h^0(B,\OO_B(L))<4$, so Lemma~\ref{lem:allsections} does not apply, and Theorem~\ref{thm:cubicblock} concerns the generalized sections of the form $\ell\cdot y$.}
\label{tab:flopcounts}
\end{table}

For instance, we have 
\[
\int_{\PP^2\times\PP^2}(H_1+H_2)^4=6, \qquad
\int_{\PP^2\times(\PP^1)^2}(H+T_1+T_2)^4=12, \qquad
\int_{(\PP^1)^4}(T_1+\cdots+T_4)^4=24. 
\]
For Example~2, we have $(2H)^4=0$ on $\PP^3\times\PP^1$, so the transition is an isomorphism.

%%%%%%%%%%%%%%%%%%%%%%%%%%%%%%%%%%%%%%%%%%%%%%%%%%%%%%%%%%%%%%%%%%%%%%%%%%%%%%%%%%%%%%%%%%%%%%
\section{The opposite phase: topology, toric realizations, and periods}

\subsection{Uniform formulas}
\label{sec:uniform}

Let us take a closer look at the opposite phase $X_-$.
The zero-locus description \eqref{eq:generalminus} determines the classical invariants of both phases directly from data on a smooth projective fourfold $B$.
In this subsection, $G$ denotes the hyperplane class of $\PP^2_y$, and $c_i=c_i(B)$.
Divisors on $B$ are pulled back to $X_\pm$, and we identify $\operatorname{Pic}(X_+)$ with $\operatorname{Pic}(X_-)$ via the birational map, which is an isomorphism in codimension 1.
\begin{proposition}
\label{prop:uniformtopology}
Under the assumptions of Lemma~\ref{lem:cubicsmooth}, we have the following.
\begin{enumerate}
\item In $H^*(\PP^2\times B,\ZZ)$, we have
\[
[X_-]=G^2(2L+3R)+G(L^2+2LR)+L^2R .
\]
Hence, for $D,D',D''\in\operatorname{Pic}(B)$, we have 
\[
G^3=0,\qquad G^2D=\int_B L^2RD,
\]
\[
GDD'=\int_B (L^2+2LR)DD',\qquad DD'D''=\int_B (2L+3R)DD'D''
\]
on $X_-$.
\item The second Chern class of $X_-$ is the restriction of
$c_2-L^2-2LR+2GL+GR-2G^2$, so that
\begin{align*}
c_2(X_-)\cdot G&=\int_B \bigl(-L^4-2L^3R-3L^2R^2+(L^2+2LR)c_2\bigr),\\
c_2(X_-)\cdot D&=\int_B \bigl(-4L^2R-4LR^2+(2L+3R)c_2\bigr)D .
\end{align*}
\item The Euler number is
\[
e(X_\pm)=\int_B \Bigl(2L^4+12L^3R+14L^2R^2+4LR^3 -(4L^2+8LR+3R^2)\,c_2+(2L+3R)\,c_3\Bigr).
\]
\item If $N_{\mathrm{flop}}>0$, let $C_+\subset X_+$ be one of the flopping curves, so that $F\cdot C_+=1$ and $D\cdot C_+=0$ for $D\in\operatorname{Pic}(B)$, where $F$ is the pullback of the hyperplane class of the cubic $\PP^1$, and $G=L-F$ by \eqref{eq:Gclass}.
Then for every $\Delta\in\operatorname{Pic}(X_\pm)$, we have 
\[
\Delta^3\big|_{X_+}=\Delta^3\big|_{X_-}+N_{\mathrm{flop}}(\Delta\cdot C_+)^3, \qquad
c_2(X_+)\cdot\Delta=c_2(X_-)\cdot\Delta-2N_{\mathrm{flop}}(\Delta\cdot C_+),
\]
and the genus-zero Gopakumar--Vafa invariant of the class $[C_+]$ on $X_+$ is $N_{\mathrm{flop}}$.
\end{enumerate}
\end{proposition}

\begin{proof}
From \eqref{eq:Qgeneral}, $c(\mathcal Q)=(1-G)^{-2}=1+2G+3G^2$, hence we have 
\[
c(\mathcal Q\boxtimes\OO_B(L))=1+(2G+2L)+(3G^2+2GL+L^2),
\]
and $[X_-]=c_2(\mathcal Q\boxtimes\OO_B(L))\,(G+R)$, and the formula in assertion 1 follows from $G^3=0$.
The intersection numbers follow by integrating against $G^2$, which is the class of a point of $\PP^2$.
Assertions 2 and 3 follow from the normal bundle sequence, which implies
\[
c(X_-)=\frac{(1+G)^3\,c(B)}{c(\mathcal Q\boxtimes\OO_B(L))\,(1+G+R)}\Big|_{X_-}
\]
with $c_1(B)=2L+R$. 
Expanding to degree 3 and integrating against $[X_-]$, we obtain the formulas.
Since $X_+$ and $X_-$ are birational smooth projective threefolds with trivial canonical class, they have the same Betti numbers \cite{BatyrevBetti}, and hence $e(X_+)=e(X_-)$.
Assertion 4 follows from the standard local calculation \cite{LLW} of the change of the cubic form and of $c_2$ under a simultaneous Atiyah flop of $N$ disjoint curves. On the common blowup, 
$\Delta$ pulls back from both sides, and the two pullbacks differ by $(\Delta\cdot C_+)$ times the exceptional divisor $E\cong\PP^1\times\PP^1$, with $E^3=2$ and $\OO_E(E)\cong\OO_{\PP^1\times\PP^1}(-1,-1)$.  
Every curve of class $[C_+]$ has intersection
number zero with $\operatorname{Pic}(B)$, so it lies in a fiber of $X_+\to B$. 
By the proof of Theorem~\ref{thm:cubicblock}(2) these fibers are finite except for the curves $C_x$.
Each $(-1,-1)$-curve contributes $1$ to the genus-zero Gopakumar--Vafa invariant of its class, and the preceding argument excludes other contributions.
\end{proof}

\begin{remark}
It is worth noting that Proposition~\ref{prop:uniformtopology} reproduces the data of every example.
For the non-Gorenstein case ($B=\PP^4$, $L=2H$, $R=H$), we have $G^2H=4$, $GH^2=8$, $H^3=7$, $c_2\cdot(H,G)=(46,36)$ and $e=-88$. 
Moreover, by assertion 4 with $\Delta=G$ and $\Delta\cdot C_+=-1$ we have $G^3=-16$ on $X_+$ and $0$ on $X_-$, and $c_2(X_+)\cdot G=36+32=68$, as in \cite{KanCompanion}. 
For Example~2, it reproduces the cubic form of Section~\ref{sec:examples12}, $c_2\cdot(H,G,T)=(36,48,24)$ and $e=-56$. 
For Example~3, it reproduces the cubic form of Section~\ref{sec:example3}, $c_2\cdot(H,G)=(62,36)$ and $e=-168$, and for Example~1 it reproduces $e=-104$.
For Example~4, it reproduces the cubic form of Section~\ref{sec:example4}, $c_2\cdot(H_1,H_2,G)=(36,36,36)$, $e=-78$ and $G^3|_{X_{4,+}}=-6$.
\end{remark}

\begin{proposition}
\label{prop:picard}
In the setting of Lemma~\ref{lem:cubicsmooth}, assume that $L$ and $R$ are ample and that $h^{1,0}(B)=h^{2,0}(B)=0$.
Then we have 
\[
H^2(X_\pm,\ZZ)\cong\operatorname{Pic}(X_\pm)\cong\ZZ F\oplus\operatorname{Pic}(B), \qquad h^{1,1}(X_\pm)=1+\rho(B),
\]
and on $X_-$ we may replace $F$ by $G=L-F$.
\end{proposition}

\begin{proof}
Let $Y_4=Z(\mathcal Q\boxtimes\OO_B(L))\subset\PP^2_y\times B$, the minus quotient of the auxiliary system without the equation $\ell\cdot y=0$.
It is a smooth fourfold by Bertini's theorem, and as in Theorem~\ref{thm:cubicblock}(1) it is isomorphic to $M$ outside $\PP^2\times Z_B(f)$ (resp. $\PP^1\times Z_B(f)$).
These loci have codimension 2 in $Y_4$ and 3 in $M$.
Since $L$ is ample, $M$ is a smooth ample divisor in the fivefold $\PP^1\times B$, so by the Lefschetz theorem $H^2(M,\ZZ)\cong H^2(\PP^1\times B,\ZZ)=\ZZ F\oplus H^2(B,\ZZ)$, and $h^{1,0}(M)=h^{2,0}(M)=0$.
These Hodge numbers are birational invariants, so they also vanish for $Y_4$, and by the exponential sequence $H^2(Y_4,\ZZ)\cong\operatorname{Pic}(Y_4)\cong\operatorname{Pic}(M)$, 
the last isomorphism because $Y_4$ and $M$ are isomorphic in codimension 1 and smooth.
Finally $X_-=Y_4\cap Z(\ell\cdot y)$, and $\ell\cdot y$ is a section of $\OO_{\PP^2}(1)\boxtimes\OO_B(R)$, which is ample on $\PP^2\times B$.
By the Lefschetz theorem for the ample divisor $X_-$ of $Y_4$, we have $H^2(Y_4,\ZZ)\cong H^2(X_-,\ZZ)$.
The statement for $X_+$ follows because $X_\pm$ are isomorphic in codimension 1.
\end{proof}

For the non-Gorenstein case ($L=2H$, $R=H$ on $\PP^4$) and for Example~3 ($L=H$, $R=3H$ on $\PP^4$) we have $h^{1,1}(X_\pm)=2$, and for Example~4 below ($L=R=H_1+H_2$ on $\PP^2\times\PP^2$) we have $h^{1,1}(X_\pm)=3$.

\begin{proposition}
\label{prop:Rzero}
Under the assumptions of Lemma~\ref{lem:cubicsmooth}, suppose that $R=0$, so that $\ell_1,\ell_2,\ell_3$ are constants.
Then for general data both $X_+$ and $X_-$ are the common zero loci of 2 sections of $\OO(1)\boxtimes\OO_B(L)$ on $\PP^1\times B$, that is, ordinary CICYs of type
\[
\left[\begin{array}{c|cc}\PP^1&1&1\\ B&L&L\end{array}\right].
\]
They are the two small resolutions of the determinantal hypersurface $\{f_0f_3-f_1f_2=0\}\subset B$, which has $N_{\mathrm{flop}}$ ordinary double points.
Moreover, if $L$ is ample and $h^{1,0}(B)=h^{2,0}(B)=0$, then $\operatorname{Pic}(X_\pm)=\ZZ F\oplus\operatorname{Pic}(B)$.
\end{proposition}

\begin{proof}
By \eqref{eq:cubicsyzygy}, the fiber of the kernel at $[u:v]$ consists of the cubic forms vanishing at that point, so $y=(y_1,y_2,y_3)$ is identified with the quadratic form $y_1u^2+y_2uv+y_3v^2$ and $\OO(-1)^{\oplus3}\cong\Sym^2V^\vee\otimes\OO(-1)$, where $V^\vee=\langle u,v\rangle$.
A constant functional $\ell$ on the $y$-space is thus a quadratic form in $(u,v)$, and for general $\ell$ it has 2 distinct roots.
After a change of coordinates on $\PP^1$ moving the roots to $0$ and $\infty$, which replaces the $f_i$ by other general sections of $\OO_B(L)$, we may assume $\ell\cdot y=y_2$.
Then the auxiliary equations become
\[
-vy_1=f_0,\qquad uy_1=f_1,\qquad -vy_3=f_2,\qquad uy_3=f_3,
\]
and eliminating $y_1,y_3$ we obtain $uf_0+vf_1=uf_2+vf_3=0$ in $\PP^1\times B$, while eliminating $(u,v)$ we obtain $y_3f_0-y_1f_2=y_3f_1-y_1f_3=0$ in $\PP^1_{[y_1:y_3]}\times B$.
Therefore $X_+$ and $X_-$ are the loci where $(s,t)$ lies in the kernel of $\left(\begin{smallmatrix}f_0&f_1\\f_2&f_3\end{smallmatrix}\right)$ (resp. of its transpose), and both map onto the determinantal hypersurface, isomorphically outside the common zero locus of the 4 entries, over which the fibers are $C_x$ (resp. $C'_x$).
For general sections of the globally generated $\OO_B(L)$ the singular locus of the determinantal hypersurface consists of these $\int_B L^4$ points, which are ordinary double points.
Finally, if $L$ is ample, then $\OO(1)\boxtimes\OO_B(L)$ is ample on the fivefold $\PP^1\times B$, and by the Lefschetz theorem for the complete intersections $X_\pm$ we have $\operatorname{Pic}(X_\pm)=\operatorname{Pic}(\PP^1\times B)$.
\end{proof}

%%%%%%%%%%%%%%%%%%%%%%%%%%%%%%%%%%%%%%%%%%%%%%%%%%%%%%%%%%%%%%%%%%%%%%%%%%%%%%%%%%%%%%%%%%%%%%

\subsection{Products of projective spaces and Batyrev--Borisov duality}
\label{sec:BBduality}

For products of projective spaces, we will show that the opposite phase $X_-$ admits a toric realization. 
Let $B=\PP^{n_1}\times\cdots\times\PP^{n_k}$ such that $\sum_{i=1}^kn_i=4$ and $H_i$ the hyperplane classes. 
We write
\[
L=\sum_{i=1}^k a_iH_i,\qquad R=\sum_{i=1}^k r_iH_i,\qquad 2a_i+r_i=n_i+1,\qquad a_i,r_i\ge0,\quad L\ne0 .
\]
The auxiliary system of Section~\ref{sec:cubicblock} has Cox coordinates
\[
u,v,\qquad y_1,y_2,y_3,\qquad x^{(i)}_0,\ldots,x^{(i)}_{n_i}\ (1\le i\le k),
\]
graded by $\ZZ\oplus\operatorname{Pic}(B)\cong\ZZ^{1+k}$ with
\[
\deg u=\deg v=(1,0),\qquad \deg y_j=(-1,L),\qquad \deg x^{(i)}_l=(0,H_i).
\]
Characters of the form $\vartheta=(-s,sL+c)$ with $s>0$ and $c=\sum_{i=1}^kc_iH_i$, $c_i>0$, form an open cone $\mathfrak C_-^{\mathrm{tot}}$, on which $y=0$ is unstable, so the GIT quotients realize the opposite phase of the transition.
The walls inside $\mathfrak C_-^{\mathrm{tot}}$ are the hyperplanes $c_i/a_i=c_j/a_j$ ($a_i,a_j>0$), 
because the cone spanned by the weights of the coordinates other than $x^{(i)}$ contains $\vartheta$ if and only if $c_i/a_i\le c_j/a_j$ for all $j$ with $a_j>0$ (and $a_i>0$).
For $\vartheta$ in a chamber $\mathfrak C_{-,i_0}$, where $i_0$ is the unique minimizer of $c_i/a_i$, let $\mathcal A_{-,i_0}$ denote the quotient.
It is a projective simplicial toric eightfold, and $X_-\subset\mathcal A_{-,i_0}$
is the common zero locus of sections of the sheaves $\OO_{\mathcal A_{-,i_0}}(L_j)$, with divisor classes
\begin{equation}
L_1=L_2=L_3=L_4=(0,L), \ \ \ L_5=(-1,L+R), \ \ \ 
\sum_{j=1}^5L_j=(-1,5L+R)=-K_{\mathcal A_{-,i_0}} .
\label{eq:fiveclassesgeneral}
\end{equation}

\begin{lemma}\label{lem:stabilizers}
In $\mathcal A_{-,i_0}$ the only points with nontrivial stabilizer lie on the stratum $\{x^{(i_0)}=0\}$.
There the stabilizer is $\mu_{a_{i_0}}$, acting with weight 1 on the $n_{i_0}+1$ transverse coordinates $x^{(i_0)}_l$ and trivially on all other coordinates.
Therefore we have the following.
\begin{enumerate}
\item $\mathcal A_{-,i_0}$ is smooth away from this stratum, and along it
has quotient singularities of type $\frac1{a_{i_0}}(1,\ldots,1)$ with $n_{i_0}+1$ entries;
\item a class $(d,\sum_{i=1}^kd_iH_i)$ is Cartier on $\mathcal A_{-,i_0}$ if and
only if $a_{i_0}\mid d_{i_0}$;
\item in particular $L_1,\ldots,L_4$ are always Cartier, while $L_5$ and
$-K_{\mathcal A_{-,i_0}}$ are Cartier, i.e. $\mathcal A_{-,i_0}$ is Gorenstein, if and only if $a_{i_0}\mid r_{i_0}$.
\end{enumerate}
\end{lemma}

\begin{proof}
The stabilizer of a point is the subgroup of $(\CC^*)^{1+k}$ acting trivially on its nonzero coordinates: $(u,v)\ne0$ forces $\lambda=1$, $x^{(j)}\ne0$ forces $\mu_j=1$, and $y\ne0$ forces $\lambda^{-1}\mu^L=1$.
Hence a point with nontrivial stabilizer has $u=v=0$ or $x^{(i)}=0$ for some $i$.
Since $\vartheta$ lies in the cone spanned by the weights of the nonzero coordinates of a semistable point, $u=v=0$ forces $y\ne0$ and $x^{(j)}\ne0$ for all $j$ (otherwise $c_j=0$), so the stabilizer is trivial, while $x^{(i)}=0$ forces $(u,v)\ne0$, $y\ne0$, $x^{(j)}\ne0$ for $j\ne i$, and $i=i_0$ by the wall description above.
In the latter case the stabilizer is $\{\mu_{i_0}\mid\mu_{i_0}^{a_{i_0}}=1\}=\mu_{a_{i_0}}$, acting on $x^{(i_0)}_l$ by $\mu_{i_0}$ and trivially on the other coordinates.
Assertion 1 follows, since a simplicial toric variety with trivial stabilizers is smooth and the normal slice to the stratum is $\CC^{n_{i_0}+1}$ with this action.
A class is Cartier at a point if and only if the corresponding character is trivial on the stabilizer; the character $(d,\sum_{i=1}^kd_iH_i)$ restricts to $\mu_{i_0}\mapsto\mu_{i_0}^{d_{i_0}}$, and assertion 2 follows.
Assertion 3 follows, since the $H_{i_0}$-coefficients of $(0,L)$, $L_5$ and $-K_{\mathcal A_{-,i_0}}$ are $a_{i_0}$, $a_{i_0}+r_{i_0}$ and $5a_{i_0}+r_{i_0}$.
\end{proof}

\begin{proposition}
\label{prop:BBduality}
In the above setting, assume that
\begin{equation}
a_i\mid r_i\qquad\text{for every $i$ with $a_i\ge2$},
\label{eq:gorcondition}
\end{equation}
and that the data are general.
Then we have the following.
\begin{enumerate}
\item for every $i_0$ with $r_{i_0}/a_{i_0}=\min_{a_j>0}r_j/a_j$, 
the quotient $\mathcal A_{-,i_0}$ is a projective Gorenstein toric variety on which $-K_{\mathcal A_{-,i_0}}$ is Cartier, nef and big, its anticanonical model $\PP_\Delta$ is a Gorenstein toric Fano variety, and the parts $E_1=\{u,y_1\}$, $E_2=\{v,y_2\}$, $E_3$ and $E_4$ each consisting of $a_i$ of the coordinates $x^{(i)}_l$ for every $i$, 
and $E_5=\{y_3\}\cup\{\text{remaining }x^{(i)}_l\}$ define a nef partition of $-K_{\PP_\Delta}$ with classes
\eqref{eq:fiveclassesgeneral};
\item the Calabi--Yau threefolds $X_+$ and $X_-$ have the Hodge
numbers of a general Calabi--Yau complete intersection $V$ of this nef partition.
\end{enumerate}
\end{proposition}

\begin{proof}
For assertion 1, by Lemma~\ref{lem:stabilizers} and \eqref{eq:gorcondition}, all 5 classes and $-K$ are Cartier on every $\mathcal A_{-,i_0}$.
For $a_i\ge1$, a point with $x^{(i)}=0$ has $\mathbf f=0$, hence $y=0$ by injectivity of $D(u,v)$ on the semistable locus, 
and for $a_i=0$ the points with $x^{(i)}=0$ are unstable, and so the semistable locus of $Z$ is the same for all chambers, and $X_-$ lies in every $\mathcal A_{-,i_0}$.
The class $(0,L)$ lies on every wall $c_i/a_i=c_j/a_j$, hence in the closure of every chamber, so $L_1,\ldots,L_4$ are nef on every $\mathcal A_{-,i_0}$.
The classes $L_5=(-1,L+R)$ and $-K=(-1,L)+(0,4L+R)$ correspond to $c=R$ and $c=4L+R$, for which $c_i/a_i=r_i/a_i$ and $4+r_i/a_i$, 
and so they lie in $\overline{\mathfrak C_{-,i_0}}$ if and only if $r_{i_0}/a_{i_0}$ is minimal, and they are nef on those $\mathcal A_{-,i_0}$.
The class $-K$ lies in the open cone $\mathfrak C_-^{\mathrm{tot}}$, so it is big, and it is ample on the quotient $\PP_\Delta$ at the character $\vartheta=-K$, to which these $\mathcal A_{-,i_0}$ map by small crepant morphisms.
Since $-K$ is Cartier and nef on $\mathcal A_{-,i_0}$, the vertices of
\[
\Delta=\{m\mid \langle m,v_\rho\rangle\ge-1\ \text{for all rays }\rho\}
\]
are the lattice points $m_\sigma$ attached to maximal cones, so $\Delta$ is a lattice polytope with $0$ in its interior, and its polar $\Delta^\circ=\operatorname{conv}\{v_\rho\}$ is a lattice polytope, 
and hence $\Delta$ is reflexive and $\PP_\Delta$ is Gorenstein Fano.
Each coordinate hyperplane $\{x_\rho=0\}$ contains $\vartheta$-stable points for $\vartheta=-K$, so every $v_\rho$ is a vertex of $\Delta^\circ$ and the $9+k$ toric divisors of $\mathcal A_{-,i_0}$ coincide with those of $\PP_\Delta$.
The parts $E_j$ partition these rays, and $\sum_{\rho\in E_j}\deg x_\rho=L_j$ by the choice of parts, using $n_i+1=2a_i+r_i$:
\[
(1,0)+(-1,L)=(0,L), \ \ \ \sum_{i=1}^k a_i(0,H_i)=(0,L), \ \ \
(-1,L)+\sum_{i=1}^kr_i(0,H_i)=(-1,L+R).
\]
Therefore $(E_1,\ldots,E_5)$ is a nef partition in the sense of \cite{BB}.

For assertion 2, since $\mathcal A_{-,i_0}\to\PP_\Delta$ is small, the spaces of sections of $\OO_{\mathcal A_{-,i_0}}(L_j)$ and $\OO_{\PP_\Delta}(L_j)$ coincide, 
and a general complete intersection $V\subset\mathcal A_{-,i_0}$ is the strict transform of a general (i.e. $\Delta$-regular) complete intersection $V_\Delta$ of the nef partition, with $V\to V_\Delta$ crepant.
The $L_j$ are base-point free, so $V$ is smooth away from the singular stratum of Lemma~\ref{lem:stabilizers}. 
That stratum has dimension $8-(n_{i_0}+1)$, and $a_{i_0}\ge2$ forces $n_{i_0}+1=2a_{i_0}+r_{i_0}\ge4$, so a general codimension-5 $V$ misses it.
Therefore $V$ is smooth.
The threefold $X_-$ is also a complete intersection of sections of $\OO_{\mathcal A_{-,i_0}}(L_j)$: the sections of $\OO_{\mathcal A_{-,i_0}}(L_1)$ have the form $f(x)+(\text{bilinear in }(u,v)\text{ and }y)$, and the equations $D(u,v)y=\mathbf f$, $\ell\cdot y=0$ are such sections.
Moreover $X_-$ misses the singular stratum, since $x^{(i_0)}=0$ forces $\mathbf f=0$ and then $y=0$ by injectivity of $D(u,v)$, and so $X_-$ is smooth.
The complete intersections that are smooth of dimension 3 form a nonempty Zariski open subset of an irreducible parameter space, hence a connected one, 
so $X_-$ is a smooth fiber of a smooth projective family containing $V$, and $h^{p,q}(X_-)=h^{p,q}(V)$.
Since $X_+$ and $X_-$ are birational Calabi--Yau threefolds, they have the same Betti numbers \cite{BatyrevBetti}, and hence the same Hodge numbers.
\end{proof}

Note that the family of all complete intersections of the nef partition is in general larger than the family of auxiliary systems. 
For example, when $R-L$ is effective, the sections of $\OO_{\mathcal A_{-,i_0}}(L_5)$ contain additional terms $u\,y_jy_l\,g(x)$. 

\begin{theorem}\label{thm:BBmirror}
Assume \eqref{eq:gorcondition}.
Let $V^*$ be a general complete intersection of the dual nef partition of Proposition~\ref{prop:BBduality} in the dual Gorenstein toric Fano variety, and let $Y$ be its strict transform in a maximal projective crepant partial desingularization.
Then $Y$ is a Calabi--Yau threefold, $Y\to V^*$ is crepant, and
\[
h^{p,q}(Y)=h^{3-p,q}(X_+)\qquad(0\le p,q\le3);
\]
in particular $h^{1,1}(Y)=h^{2,1}(X_+)$ and $h^{2,1}(Y)=h^{1,1}(X_+)$.
\end{theorem}

\begin{proof}
The quotient singularities $\frac1{a_{i_0}}(1,\ldots,1)$ of Lemma~\ref{lem:stabilizers} have $n_{i_0}+1=2a_{i_0}+r_{i_0}>a_{i_0}$ entries, hence are terminal, so the small crepant morphism $\mathcal A_{-,i_0}\to\PP_\Delta$ is a maximal projective crepant partial desingularization in the sense of \cite{Batyrev}, and $V$ is the corresponding desingularization of $V_\Delta$.
By \cite{BB}, the dual partition is again a nef partition of the dual reflexive polytope, so a general complete intersection $V^*$ of it is a threefold with Gorenstein singularities and trivial canonical class.
A maximal projective crepant partial desingularization of the dual Gorenstein toric Fano variety has terminal Gorenstein singularities, whose locus has codimension at least 4 \cite{Batyrev}.
Since $\pi_*\OO=\OO$ for the desingularization $\pi$, the sections of the pulled-back classes are the pullbacks of those on the dual Fano variety, and for general $V^*$ the common zero locus of the pulled-back sections is the strict transform $Y$ (the exceptional locus contributes no component, by dimension count).
A general such $Y$ therefore misses the singular locus, so $Y$ is smooth and $Y\to V^*$ is crepant.
Therefore $Y$ is smooth, $K_Y=0$ by adjunction, and $h^{0,q}(Y)=0$ for $0<q<3$ by \cite{BB}; hence $Y$ is a Calabi--Yau threefold.
By Batyrev--Borisov duality \cite{BB}, we have
\[
h^{1,1}(Y)=h^{1,2}(V),\qquad h^{1,2}(Y)=h^{1,1}(V).
\]
By Proposition~\ref{prop:BBduality}(2), these are $h^{1,2}(X_+)$ and $h^{1,1}(X_+)$, respectively.
Since $V$, $Y$, and $X_+$ are Calabi--Yau threefolds, these equalities determine the full Hodge diamonds and give the stated formula.
\end{proof}

Among the 18 configurations of Table~\ref{tab:flopcounts}, condition \eqref{eq:gorcondition} fails only for the non-Gorenstein case $B=\PP^4$, $(a,r)=(2,1)$ of \cite{KanCompanion}, since by Lemma~\ref{lem:stabilizers} it can fail only for a factor with $a_i\ge2$ and $a_i\nmid r_i$, and the only such pair in the proof of Proposition~\ref{prop:classification} is $(2,1)$ on $\PP^4$.
Hence Theorem~\ref{thm:BBmirror} applies to the other 17 configurations, in particular to Examples~1--4.

%%%%%%%%%%%%%%%%%%%%%%%%%%%%%%%%%%%%%%%%%%%%%%%%%%%%%%%%%%%%%%%%%%%%%%%%%%%%%%%%%%%%%%%%%%%%%%

\subsection{Fibrations and the fundamental period}
\label{sec:fibperiod}

In this subsection, we assume that $B=\PP^{n_1}\times\cdots\times\PP^{n_k}$ as in Section~\ref{sec:BBduality}, with $L=\sum_{i=1}^ka_iH_i\neq0$ and $R=\sum_{i=1}^kr_iH_i$.

\begin{proposition}
\label{prop:twofibrations}
\begin{enumerate}
\item The projection $X_+\to\PP^1_{[u:v]}$ is a K3 fibration. 
Its general fiber is the smooth complete intersection $S=Z_B\bigl(P(u,v),\ \ell\cdot y\bigr)$ of divisors of classes $L$ and $L+R$. 
\item If $L^2R\neq0$, the projection $X_-\to\PP^2_y$ is surjective and its general fiber is the smooth complete intersection $E_y$ of divisors of classes $L,L,R$ in $B$, and hence is a disjoint union of genus-one curves, with
\[
D\cdot E_y=\int_B L^2RD\qquad(D\in\operatorname{Pic}(B)),\qquad G\cdot E_y=0.
\]
In particular, if $E_y$ is connected, the index of the fibration divides $\gcd_{1\le i\le k}\int_B L^2RH_i$.
If $R=0$, the $\ell_j$ are constants, so the image of $X_-\to\PP^2_y$ is a line, and the fibers over it are the complete intersections of 2 divisors of class $L$ in $B$, which are K3 surfaces as in assertion 1.
\item If the $i$-th factor of $B$ is a $\PP^2$, write $B=\PP^2\times B'$.
Then the fibers of the projection $X_-\to\PP^2$ to that factor are the zero loci in $\PP^2_y\times B'$ of the restriction of $(\mathcal Q\boxtimes\OO_B(L))\oplus(\OO(1)\boxtimes\OO_B(R))$. 
A smooth nonempty fiber has trivial canonical divisor, hence is a disjoint union of genus-one curves.
Its class is $H_i^2$, with
\[
D\cdot H_i^2=\int_B (2L+3R)H_i^2D\quad(D\in\operatorname{Pic}(B)),\qquad G\cdot H_i^2=\int_B (L^2+2LR)H_i^2 .
\]
\end{enumerate}
\end{proposition}

\begin{proof}
For assertion 1, over a fixed $[u:v]$ the auxiliary vector $y$ is determined linearly by $\mathbf f$, so its components are sections of $\OO_B(L)$ and $\ell\cdot y$ is a section of $\OO_B(L+R)$.
The general fiber is smooth by generic smoothness, and $K_S=(K_B+2L+R)|_S=0$ by adjunction and
\eqref{eq:CYcondition}.  In the Koszul complex
$0\to\OO_B(-2L-R)\to\OO_B(-L)\oplus\OO_B(-L-R)\to\OO_B\to\OO_S\to0$ the middle terms have no cohomology: choose $i$ such that $a_i\ge1$, and
then $1\le a_i\le n_i$ and $1\le a_i+r_i=n_i+1-a_i\le n_i$, so the corresponding factor $H^*(\PP^{n_i},\OO(-a_i))$, resp.
$H^*(\PP^{n_i},\OO(-a_i-r_i))$, vanishes and the K\"unneth formula kills the whole group.
Since $-2L-R=K_B$, we have $h^i(\OO_B(-2L-R))=\delta_{i4}$ by Serre duality, and hence $h^0(\OO_S)=h^2(\OO_S)=1$ and $h^1(\OO_S)=0$, so $S$ is a K3 surface.

For assertion 2, over $[y]\in\PP^2$ the rank-3 bundle of \eqref{eq:generalminus} restricts to $\OO_B(L)^{\oplus2}\oplus\OO_B(R)$, since $\mathcal Q_y\cong\CC^2$.
Hence the fiber is $E_y$, which is smooth for general $y$, and $K_{E_y}=0$ by the adjunction formula. 
Surjectivity holds because $E_y$ is a complete intersection of base-point-free classes whose class $L^2R$ is nonzero, hence nonempty.
The intersection numbers are those of Proposition~\ref{prop:uniformtopology}(1).
If $R=0$ the section $\ell\cdot y$ has constant coefficients and cuts out a line in $\PP^2_y$.

For assertion 3, restricting to $\{x\}\times B'$ kills $H_i$, and $-K_B=\sum_{j=1}^k(n_j+1)H_j=2L+R$.
Hence the restricted bundle has determinant class $3G-K_{B'}=-K_{\PP^2_y\times B'}$.
The intersection numbers follow from Proposition~\ref{prop:uniformtopology}(1).
\end{proof}

\begin{lemma}
\label{lem:index}
Let $X\to S$ be a genus-one fibration of a Calabi--Yau threefold $X$ with fiber class $\phi$, and assume that $\operatorname{Pic}(X)/\mathrm{tors}$ is generated by $D_1,\ldots,D_\rho$.
Then the index of the fibration is the greatest common divisor of $D_1\cdot\phi,\ldots,D_\rho\cdot\phi$.
\end{lemma}

\begin{proof}
The proof is parallel to that of \cite[Prop.~2.8]{KanCompanion}.
A divisor class of positive degree on the generic fiber, a curve of genus one, is effective by the Riemann--Roch theorem, so a combination of the $D_i$ of degree equal to the greatest common divisor is realized by a multisection.
Conversely, the closure of a closed point of the generic fiber is a Cartier divisor on the smooth threefold $X$, so its degree is divisible by that common divisor.
\end{proof}

Each toric phase in the examples below carries a genus-one fibration, and the hypothesis of Lemma~\ref{lem:index} holds in all examples: for the non-Gorenstein case and Examples~3 and~4 by Proposition~\ref{prop:picard}, 
for Example~1 by Proposition~\ref{prop:Rzero}, and for Example~5 by the Lefschetz theorem for ample hypersurfaces in smooth toric fourfolds.
For Example~2 we give a direct argument in Section~\ref{sec:examples12}.

\begin{proposition}
Let us use the basis $(H_1,\ldots,H_k,G)$ of $\operatorname{Pic}(X_-)\otimes\QQ$ and corresponding coordinates $(z_1,\ldots,z_k,w)$.
Then the Hori--Vafa model associated with the partition $E_1,\ldots,E_5$ of Proposition~\ref{prop:BBduality} (which makes sense without the Gorenstein condition) has fundamental period
\begin{equation}
\begin{aligned}
\Phi(z,w)&=\sum_{n\in\ZZ_{\ge0}^k}\ \sum_{m=0}^{a\cdot n}
\frac{(a\cdot n)!^4\,(r\cdot n+m)!}
{(a\cdot n-m)!^2\,m!^3\,\prod_{i=1}^k(n_i!)^{2a_i+r_i}}\ z^nw^m,\\
&a\cdot n=\sum_{i=1}^ka_in_i,\qquad r\cdot n=\sum_{i=1}^kr_in_i .
\end{aligned}
\label{eq:uniformperiod}
\end{equation}
Its restriction to $w=0$,
\[
\Phi(z,0)=\sum_{n\in\ZZ_{\ge0}^k}\frac{(a\cdot n)!^2\,(r\cdot n)!}{\prod_{i=1}^k(n_i!)^{2a_i+r_i}}\,z^n,
\]
is the hypergeometric series attached to the complete intersection $E_y$ of classes $L,L,R$ in $B$, i.e. to the fibers of Proposition~\ref{prop:twofibrations}(2).
\end{proposition}

\begin{proof}
A curve class $\beta$ such that $H_i\cdot\beta=n_i$ and $G\cdot\beta=m$ satisfies $D_u\cdot\beta=D_v\cdot\beta=a\cdot n-m$, $D_{y_j}\cdot\beta=m$ and $D_{x^{(i)}_l}\cdot\beta=n_i$, since $\deg u=\deg v=(1,0)=L-G$, $\deg y_j=G$ and $\deg x^{(i)}_l=H_i$; moreover $L_1,\ldots,L_4\cdot\beta=a\cdot n$ and $L_5\cdot\beta=(G+R)\cdot\beta=m+r\cdot n$.
The fundamental period of the Hori--Vafa model is the standard series
\[
\Phi=\sum_{\substack{\beta\,:\,D_\rho\cdot\beta\ge0\\ \text{for all }\rho}}\frac{\prod_{j=1}^5(L_j\cdot\beta)!}
{\prod_{\rho}(D_\rho\cdot\beta)!}\,q^\beta
\]
\cite{HKTY,Givental}, which is \eqref{eq:uniformperiod}.
Setting $m=0$, the factors $(a\cdot n)!^2$ in the denominator coming from $u,v$ cancel 2 of the 4 factors $(a\cdot n)!$ in the numerator.
\end{proof}

Formula \eqref{eq:uniformperiod} specializes to the periods of Examples~2, 3 and 4 in Section~\ref{sec:mirrors}, and to that of the non-Gorenstein case in \cite{KanCompanion}.
For Example~1 ($R=0$) we use in Section~\ref{sec:mirrors} the period attached to the ordinary CICY presentation of Proposition~\ref{prop:Rzero} instead; the two hypergeometric series come from different toric ambient spaces of the same threefold, and they agree on $w=0$.

%%%%%%%%%%%%%%%%%%%%%%%%%%%%%%%%%%%%%%%%%%%%%%%%%%%%%%%%%%%%%%%%%%%%%%%%%%%%%%%%%%%%%%%%%%%%%%
\section{Examples 1 and 2: two-quadric fibrations}
\label{sec:examples12}

\subsection{Example 1: an ordinary CICY on both sides}

For $X_{1,+}$ in \eqref{eq:X1intro}, we take $B=\PP^3_H\times\PP^1_T$, $L=2H+T,$ and $R=0$. 
By  Proposition~\ref{prop:Rzero} the generalized phase is a CICY 
\begin{equation}
X_{1,+}\cong
\left[\begin{array}{c|cc}
\PP^3&2&2\\
\PP^1&1&1\\
\PP^1&1&1
\end{array}\right],
\label{eq:X1cicy}
\end{equation}
with the cubic $\PP^1$ of the generalized presentation being one of the two $\PP^1$ factors.
In other words, the negative entry of \eqref{eq:X1intro} is an artifact of the presentation.
The opposite phase has the same CICY type, and the two models are the two small resolutions of the determinantal anticanonical hypersurface $\{f_0f_3-f_1f_2=0\}\subset\PP^3\times\PP^1$.
It has $\int_B (2H+T)^4=32$ ordinary double points, and the two small resolutions differ by $32$ ordinary flops.
By Proposition~\ref{prop:Rzero}, $\operatorname{Pic}(X_{1,+})$ is generated by $H,T,F$, and together with $e(X_{1,\pm})=-104$, we obtain $(h^{1,1},h^{2,1})=(3,55)$.

Let $(H,T,F)$ be the natural divisor basis in the generalized phase.
The cubic data are
\[
H^3=2, \qquad H^2T=H^2F=HTF=4, \qquad T^2=F^2=0,
\]
\[
c_2(X_{1,+})\cdot(H,T,F)=(44,24,24). 
\]
In the toric phase put $G=2H+T-F$. 
Then, in the basis $(H,T,G)$, the cubic form $\mu_{1,-}(x,y,z):= \int_{X_{1,-}}(xH+yT+zG)^3$ and $c_2$ are given by 
\begin{equation}
\mu_{1,-}(x,y,z)
=2x^3+12x^2y+12x^2z+24xyz,
\end{equation}
\[
c_2(X_{1,-})\cdot(H,T,G)=(44,24,24). 
\]
For $1 \le i \le 4$, let
$\mathcal R_i:=H^{\mathrm{even}}(X_{i,-},\QQ)$
be the rational even cohomology ring.  
Then 
\begin{equation}
\mathcal R_1
=\QQ[H,T,G]/(T^2,G^2,2H^2-2H(T+G)+3TG).
\end{equation}
The projection to $\PP^1\times\PP^1$ has general fiber a $(2,2)$ complete intersection in $\PP^3$, hence a genus-one curve.
Its index is 4: the fiber class is $T\cdot F$, and $(aH+bT+cF)\cdot TF=4a$, with $H$ a 4-section (Lemma~\ref{lem:index}).

%%%%%%%%%%%%%%%%%%%%%%%%%%%%%%%%%%%%%%%%%%%%%%%%%%%%%%%%%%%%%%%%%%%%%%%%%%%%%%%%%%%%%%%%%%%%%%

\subsection{Example 2: the transition is an isomorphism}

For $X_{2,+}$ in \eqref{eq:X2intro}, we take $B=\PP^3_H\times\PP^1_T$, and $L=2H$ and $R=2T$.  
General quadrics $f_0,\ldots,f_3$ in $\PP^3$ have no common zero, and hence Theorem~\ref{thm:cubicblock} implies that the generalized phase is isomorphic to
\begin{equation}
X_2:=X_{2,+}\cong X_{2,-}\cong
Z\!\left((\mathcal Q\boxtimes\OO_{\PP^3}(2))
\oplus\OO_{\PP^2_y\times\PP^1_t}(1,2)\right)
\subset\PP^2_y\times\PP^3\times\PP^1_t.
\label{eq:X2bundle}
\end{equation}
Let
\[
S=\{\ell_1(t)y_1+\ell_2(t)y_2+\ell_3(t)y_3=0\}
\subset\PP^2_y\times\PP^1_t.
\]
For general $\ell_i$, the 3 quadrics $\ell_i(t)$ form a basis of $H^0(\PP^1,\OO(2))$.
For fixed $t$, the equation is a line in $\PP^2_y$, and the kernel of the evaluation map $\OO^{\oplus3}\to\OO(2)$ is $\OO(-1)^{\oplus2}$, so $S\cong\PP_{\PP^1}(\OO(-1)^{\oplus2})\cong\PP^1\times\PP^1$.
For fixed $y$, the equation is a binary quadric in $t$, so $S\to\PP^2_y$ is the double cover branched along a conic.
The map $X_2\to S$ has general fiber a $(2,2)$ curve in $\PP^3$ and is the relevant connected genus-one fibration.
Therefore $X_2\to\PP^2_y$ has 2 connected genus-one components in a general fiber.

The genus-one fibration $X_2\to S$ has index 4.
Its fiber $E$ has $H\cdot E=4$ and $G\cdot E=T\cdot E=0$.
Since $h^{1,1}(X_2)=3$ (\cite{AAGGL}; cf. Section~\ref{sec:intro}), $H,G,T$ form a basis of $N^1(X_2)_\QQ$. 
Therefore every $D\in\operatorname{Pic}(X_2)$ satisfies $D\equiv cH+aG+bT$ for some $a,b,c\in\QQ$. 
The fibration has vertical lines $\ell$ (lines in the fibers) such that $H\cdot\ell=1$ and $G\cdot\ell=T\cdot\ell=0$.
Indeed, write $X_2\subset S\times\PP^3$ as the zero locus of a section of $\mathcal K\boxtimes\OO_{\PP^3}(2)$ with $\mathcal K=\pi^*\mathcal Q$ for the double cover $\pi\colon S\to\PP^2_y$, and let $\gamma=\pi^*c_1(\OO_{\PP^2_y}(1))$, so that $c(\mathcal K)=1+2\gamma+3\gamma^2$ and $\int_S\gamma^2=2$.
Vertical lines are the zeros of the induced section of $\mathcal K\boxtimes\Sym^2\mathcal S^*$ on $S\times\Gr(2,4)$, where $\mathcal S$ is the tautological bundle, and the rank 6 equals $\dim S\times\Gr(2,4)$.
With $\alpha,\beta$ the Chern roots of $\mathcal S^*$ and
$\int_{\Gr(2,4)}p=[\alpha^3\beta^2](\alpha-\beta)p$ for symmetric $p$, the top Chern class integrates to
\[
2\,[\alpha^3\beta^2](\alpha-\beta)\,[\gamma^2]\!\!
\prod_{r\in\{2\alpha,\alpha+\beta,2\beta\}}\!\!
(r^2+2\gamma r+3\gamma^2)=192.
\]
Therefore, there are $192$ vertical lines counted with multiplicity, and in particular one exists.
Therefore $c=D\cdot\ell\in\ZZ$ and $D\cdot E=4c\in4\ZZ$, while $H$ is a 4-section.

In the original basis $(H,T,F)$, we have 
\[
\mu_{2,+}(x,y,z)=6x^3+12x^2y+12x^2z+24xyz,
\]
\[
c_2(X_2)\cdot(H,T,F)=(36,24,24), \qquad e(X_2)=-56.
\]
The basis natural from \eqref{eq:X2bundle} is $(H,G,T)$ with $G=2H-F$.
Then we have 
\begin{equation}
\mu_{2,-}(x,y,z)=6x^3+24x^2y+12x^2z+24xy^2+24xyz,
\end{equation}
\[
c_2(X_2)\cdot(H,G,T)=(36,48,24),
\]
and
\begin{equation}
\mathcal R_2
=\QQ[H,G,T]/(T^2,G^2-2GT,(2H-G)^2).
\end{equation}

The toric eightfold constructed in Section~\ref{sec:BBduality} is Gorenstein but singular in this example.
Only the $\PP^3$ factor has $a_i>0$, so $\mathfrak C_-^{\mathrm{tot}}$ is a single chamber.
By Lemma~\ref{lem:stabilizers}, its only nontrivial stabilizer is $\mu_2$ along the stratum $x=0$ of the $\PP^3$ coordinates, acting with weight 1 on the 4 transverse coordinates $x_i$.
The quotient singularity is therefore of type $\frac12(1,1,1,1)$, which is Gorenstein.
The 5 defining classes, 4 copies of $2H$ and the class $G+2T$, and the anticanonical class are Cartier, and the Calabi--Yau zero locus avoids the quotient-singular stratum.
This is the case $a=(2,0)$, $r=(0,2)$ of Lemma~\ref{lem:stabilizers}, and Proposition~\ref{prop:BBduality} and Theorem~\ref{thm:BBmirror} apply to this toric phase.

%%%%%%%%%%%%%%%%%%%%%%%%%%%%%%%%%%%%%%%%%%%%%%%%%%%%%%%%%%%%%%%%%%%%%%%%%%%%%%%%%%%%%%%%%%%%%%
\section{Example 3: one flop to a toric Fano hypersurface}
\label{sec:example3}

For $X_{3,+}$ in \eqref{eq:X3intro}, we take $B=\PP^4$, $L=H$ and $R=3H$.   
The 4 coefficients of the cubic in $[u:v]$ are general linear forms on $\PP^4$.
Their common zero locus is one point, so the transition of Theorem~\ref{thm:cubicblock} is a single ordinary flop.
Here both intermediate fourfolds of the construction are projective bundles.
The original intermediate fourfold can be written
\[
M_3=\PP_{\PP^1}
\bigl(\OO\oplus\OO(-1)^{\oplus3}\bigr), \qquad
X_{3,+}\in|-K_{M_3}|,
\]
but $-K_{M_3}=4H-F$ is not nef.
On the opposite side, the fourfold $Y_4=Z(\mathcal Q\boxtimes\OO(1))$ of the proof of Proposition~\ref{prop:picard} is a projective bundle over $\PP^2_y$. 
For fixed $y$, the condition that $\mathbf f(x)$ lie in the image of $A(y)$ consists of 2 linear equations on $\PP^4$, and hence 
\begin{equation}
Y_4=V_3=\PP_{\PP^2}
\bigl(\OO\oplus\OO(-1)^{\oplus2}\bigr), \qquad
X_{3,-}\in|-K_{V_3}|.
\end{equation}
If $G$ is the hyperplane class of the base, then $-K_{V_3}=3H+G$ is ample, and $V_3$ is a smooth toric Fano fourfold.

On the generalized side, we have 
\[
H^3=11, \qquad H^2F=4, \qquad F^2=0, \qquad c_2(X_{3,+})\cdot(H,F)=(62,24).
\]
On the toric side, with $G=H-F$, we have 
\begin{equation}
H^3=11, \quad H^2G=7, \quad HG^2=3, \quad G^3=0, \qquad
c_2(X_{3,-})\cdot(H,G)=(62,36).
\end{equation}
Therefore, in the basis $(H,G)$, we have
\begin{equation}
\mu_3(x,y)=11x^3+21x^2y+9xy^2, \qquad
\mathcal R_3
=\QQ[H,G]/(9H^2-21HG+16G^2,G^3).
\end{equation}
Proposition~\ref{prop:picard} shows that $\operatorname{Pic}(X_{3,-})$ is generated by $H$ and $G$, and together with $e(X_{3,\pm})=-168$, we obtain $(h^{1,1},h^{2,1})=(2,86)$.\footnote{
The dimension count $h^0(V_3,\OO_{V_3}(-K_{V_3}))-1-\dim\operatorname{Aut}(V_3)=105-1-18=86$ is compatible with this.} 
The plane-cubic fibration $X_{3,-}\to\PP^2$ has index 3, since the fiber class is $G^2$ and $(rH+gG)\cdot G^2=3r$ (Lemma~\ref{lem:index}).

The $(4,2)$ hypersurface in $\PP^3\times\PP^1$ has the same Hodge numbers and also a K3 fibration over $\PP^1$, but it is not deformation equivalent to $X_{3,+}$. 
This can be seen from the integral cubic form on $H^2$, which is a deformation invariant. 
In the bases $(H,F)$, where $F$ is the class of the K3 fiber, the cubic forms are
\[
2x^3+12x^2y\quad\text{for the $(4,2)$ hypersurface},\qquad 11x^3+12x^2y\quad\text{for }X_{3,+},
\]
and no unimodular change of basis identifies them as modulo $4$ they reduce to $2x^3$ and $3x^3$, whose sets of values differ.

%%%%%%%%%%%%%%%%%%%%%%%%%%%%%%%%%%%%%%%%%%%%%%%%%%%%%%%%%%%%%%%%%%%%%%%%%%%%%%%%%%%%%%%%%%%%%%
\section{\texorpdfstring{Example 4: a symmetric configuration over $\PP^2\times\PP^2$}{Example 4: a symmetric configuration over P2 x P2}}
\label{sec:example4}

\subsection{The two phases}
For $X_{4,+}$ in \eqref{eq:X4intro}, we take $B=\PP^2_{x}\times\PP^2_{x'}$ and $L=R=H_1+H_2$, where $H_1,H_2$ are the two hyperplane classes.
The 4 coefficients $f_i$ of the cubic in $[u:v]$ are general $(1,1)$-forms, and the generalized section is $\ell\cdot y$ with $\ell_j$ general $(1,1)$-forms.
The family is symmetric under exchanging the two factors of $B$.

General $(1,1)$-forms $f_0,\ldots,f_3$ on $\PP^2\times\PP^2$ meet in $\int(H_1+H_2)^4=6$ reduced points, so by Theorem~\ref{thm:cubicblock} $X_{4,+}$ is related by a simultaneous flop of 6 disjoint $(-1,-1)$-curves to
\begin{equation}
X_{4,-}=
Z\!\left((\mathcal Q\boxtimes\OO(1,1))\oplus\OO(1,1,1)\right)
\subset\PP^2_y\times\PP^2_x\times\PP^2_{x'} .
\label{eq:X4minus}
\end{equation}
Both $L$ and $R$ are ample, and hence by Proposition~\ref{prop:picard} we have
\[
\operatorname{Pic}(X_{4,\pm})=\ZZ H_1\oplus\ZZ H_2\oplus\ZZ G, \qquad G=H_1+H_2-F,
\]
and $h^{1,1}(X_{4,\pm})=3$.
By Proposition~\ref{prop:uniformtopology}, we have $e(X_{4,\pm})=-78$, and hence $(h^{1,1},h^{2,1})=(3,42)$. 

We can make the toric phase explicit by the construction of Proposition~\ref{prop:picard}. 
The fourfold $Y_4=Z(\mathcal Q\boxtimes\OO(1,1))\subset\PP^2_y\times
\PP^2_x\times\PP^2_{x'}$ is isomorphic in codimension 1 to the
intermediate fourfold $M=Z(P)\subset\PP^1\times\PP^2\times\PP^2$. 
The 6 planes $\PP^2_y\times\{x_i\}$ of $Y_4$ correspond to the 6 lines $\PP^1\times\{x_i\}$ of $M$.
The toric phase $X_{4,-}$ is the ample divisor $Y_4\cap Z(\ell\cdot y)$ of class $G+H_1+H_2$.

Here $a=r=(1,1)$, so by Lemma~\ref{lem:stabilizers} both toric ambient spaces $\mathcal A_{-,1}$ and $\mathcal A_{-,2}$ of the opposite phase, corresponding to the two chambers $c_1<c_2$ and $c_1>c_2$ of Section~\ref{sec:BBduality}, are smooth, and Proposition~\ref{prop:BBduality} and Theorem~\ref{thm:BBmirror} apply.
The proper transforms of \eqref{eq:X4minus} in the two chamber quotients are canonically isomorphic. 
In particular, the wall crossing at $c_1=c_2$ induces no birational modification of the Calabi--Yau threefold.

\subsection{Topological data}
In the basis $(H_1,H_2,G)$ of the toric phase, by Proposition~\ref{prop:uniformtopology} the nonzero intersection numbers are
\[
H_1^2H_2=H_1H_2^2=5,\quad H_1^2G=H_2^2G=3,\quad H_1H_2G=6,\quad H_1G^2=H_2G^2=3,
\]
while $H_1^3=H_2^3=G^3=0$.
Therefore we have
\begin{equation}
\mu_{4,-}(x,y,z)=15x^2y+15xy^2+9x^2z+9y^2z+36xyz+9xz^2+9yz^2,
\end{equation}
\[
c_2(X_{4,-})\cdot(H_1,H_2,G)=(36,36,36),
\]
and
\begin{equation}
\begin{aligned}
\mathcal R_4=\QQ[H_1,H_2,G]\big/\bigl(&(H_1+H_2)^2-3(H_1+H_2)G+4G^2,\\
&H_1^2-H_1H_2+H_2^2,\ (H_1-H_2)(3H_1+3H_2-5G)\bigr).
\end{aligned}
\end{equation}
The 3 quadrics are a basis of the kernel of the pairing $\operatorname{Sym}^2H^2\to(H^2)^\vee$ defined by $\mu_4$, and they form a regular sequence, so $\mathcal R_4$ has Hilbert function $(1,3,3,1)$.
The exchange of the two factors of $B$ acts on $\mathcal R_4$ by swapping $H_1$ and $H_2$. 
Therefore, the first two relations are invariant, and the third changes sign.

In the basis $(H_1,H_2,F)$ of the generalized phase, by Proposition~\ref{prop:uniformtopology}(4) with $F\cdot C_+=1$ we have
\[
\mu_{4,+}(x,y,z)=15x^2y+15xy^2+6x^2z+24xyz+6y^2z,
\]
\[
c_2(X_{4,+})\cdot(H_1,H_2,F)=(36,36,24).
\]
Note that $F^2=0$ and $c_2\cdot F=24$, which is compatible with the K3 fibration $X_{4,+}\to\PP^1$, whose fiber is the complete intersection of a $(1,1)$ and a $(2,2)$ divisor in $\PP^2\times\PP^2$.
In the toric-side basis the jump across the flop is $G^3|_{X_{4,+}}=-6$ and $c_2(X_{4,+})\cdot G=36+12=48$.

\subsection{The genus-one fibrations}
The toric phase has 3 genus-one fibrations to $\PP^2$, one of index 3 and two of index 1. 
\begin{enumerate}
\item Over $\PP^2_y$, the fiber is the complete intersection of 3
$(1,1)$-divisors in $\PP^2\times\PP^2$, a genus-one curve of bidegree $(3,3)$ (cf. Proposition~\ref{prop:twofibrations}(2)).
The fiber class is $G^2$, and $(aH_1+bH_2+cG)\cdot G^2=3a+3b$.
Since $\operatorname{Pic}(X_{4,-})$ is generated by $H_1,H_2,G$, Lemma~\ref{lem:index} shows that the index is 3.
\item Over $\PP^2_x$, the fiber is the zero locus in $\PP^2_y\times\PP^2_{x'}$
of $(\mathcal Q\boxtimes\OO(1))\oplus\OO(1,1)$ (cf. Proposition~\ref{prop:twofibrations}(3)).
It is connected, since in its Koszul complex the terms other than $\OO$ and $\OO(K)$ are acyclic by the vanishing of $H^*(\PP^2,\OO(-1))$ and $H^*(\PP^2,\OO(-2))$.
The fiber class is $H_1^2$ and $(aH_1+bH_2+cG)\cdot H_1^2=5b+3c$, so the index is 1 and the fibration has a rational section (Lemma~\ref{lem:index}).
Similarly for $\PP^2_{x'}$ by symmetry.
\end{enumerate} 

%%%%%%%%%%%%%%%%%%%%%%%%%%%%%%%%%%%%%%%%%%%%%%%%%%%%%%%%%%%%%%%%%%%%%%%%%%%%%%%%%%%%%%%%%%%%%%
\section{\texorpdfstring{Beyond $[3,-1]$~blocks}{Beyond [3,-1] blocks}}
\label{sec:beyond}

The same construction also applies to configurations without a $[3,-1]$~block. 
We will first consider the case where the negative entry can be removed.

\subsection{Disguised complete intersections}

\begin{lemma}\label{lem:disguise}
Let $A$ be a product of projective spaces, $L$ a divisor class on $A$, and $[s_0:s_1]$ the coordinates of $\PP^1$.
Consider a configuration on $\PP^1\times A$ whose first two columns are
\[
\left[\begin{array}{c|cc}
\PP^1&1&-1\\ A&L&L
\end{array}\right],
\]
followed by further columns, and let $p=s_0f_0+s_1f_1$, with $f_0,f_1\in H^0(A,\OO_A(L))$ not both zero, be the first defining section.
Then $H^0\bigl(Z(p),(\OO(-1)\boxtimes\OO_A(L))|_{Z(p)}\bigr)$ is spanned by the section $w$ determined by $(f_0,f_1)=w\,(s_1,-s_0)$, and
\[
Z(p)\cap Z(w)=\PP^1\times Z_A(f_0,f_1). 
\]
In particular, if the sections of the remaining columns on $\PP^1\times Z_A(f_0,f_1)$ are restrictions of sections on $\PP^1\times A$, 
the configuration defines the same threefold as the ordinary configuration obtained by replacing the entries $1,-1$ of the $\PP^1$ row by $0,0$.
\end{lemma}

\begin{proof}
Multiplication by $p$ induces the exact sequence
\[
0\to\OO(-2)\boxtimes\OO_A\to\OO(-1)\boxtimes\OO_A(L)\to (\OO(-1)\boxtimes\OO_A(L))|_{Z(p)}\to0 .
\]
Since $H^*(\PP^1,\OO(-1))=0$, by the K\"unneth formula we have
\[
H^0\bigl((\OO(-1)\boxtimes\OO_A(L))|_{Z(p)}\bigr)\cong H^1(\PP^1,\OO(-2))\otimes H^0(A,\OO_A)\cong\CC.
\]
On $Z(p)$ the local expressions $f_0/s_1$ and $-f_1/s_0$ agree where both are defined, because $s_0f_0+s_1f_1=0$, and they glue to a nonzero section $w$ such that $(f_0,f_1)=w\,(s_1,-s_0)$.
Where $s_0=1$ the ideal $(p,w)$ equals $(f_0+s_1f_1,f_1)=(f_0,f_1)$, and similarly where $s_1=1$.
\end{proof}

Therefore, a configuration with a single negative entry paired in this way is an ordinary CICY.
We give two examples.

\begin{enumerate}
\item For
\[
\left[\begin{array}{c|ccc}
\PP^1&1&-1&2\\ \PP^5&1&1&4
\end{array}\right]
\]
the locus $Z_A(f_0,f_1)$ is a linear $\PP^3\subset\PP^5$, and the threefold is a general divisor of bidegree $(2,4)$ in $\PP^1\times\PP^3$, the ordinary CICY with $(h^{1,1},h^{2,1})=(2,86)$.
This is the threefold compared with $X_{3,+}$ in Section~\ref{sec:example3}: the two have the same Hodge numbers but are not deformation equivalent.
Although $X_{3,+}$ also has a single negative entry, it is not an ordinary CICY in disguise.
\item For
\[
\left[\begin{array}{c|ccc}
\PP^1&1&-1&2\\ \PP^1&1&1&0\\ \PP^4&1&1&3
\end{array}\right]
\]
we have $A=\PP^1_t\times\PP^4$ and $\OO_A(L)\cong\OO_A(1,1)$.
Write $f_i=t_0a_i(x)+t_1b_i(x)$.
The threefold $W=Z_A(f_0,f_1)$ is the complete intersection of 2 $(1,1)$-divisors in $\PP^1\times\PP^4$. 
It is the projective bundle $\PP_{\PP^1}(\OO\oplus\OO(-1)^{\oplus2})$ over $\PP^1_t$, and its image in $\PP^4$ is the quadric cone $Q=\{a_0b_1-a_1b_0=0\}$ of rank 4, of which $W$ is a small resolution.
By Lemma~\ref{lem:disguise} the threefold $X$ is the ordinary CICY
\[
\left[\begin{array}{c|ccc}
\PP^1&0&0&2\\ \PP^1&1&1&0\\ \PP^4&1&1&3
\end{array}\right],
\]
an anticanonical hypersurface in $\PP^1\times W$.
A standard CICY computation shows $h^{1,1}(X)=3$ and $e(X)=-144$, hence $(h^{1,1},h^{2,1})=(3,75)$.
\end{enumerate}

When negative columns act on two different $\PP^1$ factors, as in Example~5, the eliminations interact and the result is not necessarily a complete intersection.

%%%%%%%%%%%%%%%%%%%%%%%%%%%%%%%%%%%%%%%%%%%%%%%%%%%%%%%%%%%%%%%%%%%%%%%%%%%%%%%%%%%%%%%%%%%%%%%%%%%%%

\subsection{Example 5: two negative columns and a toric Fano hypersurface}

The threefold 
\begin{equation}
X_{5,+}=\left[\begin{array}{c|cccc}
\PP^1&1&1&-1&1\\
\PP^1&1&1&1&-1\\
\PP^5&3&1&1&1
\end{array}\right]
\end{equation} has two negative columns, one on each $\PP^1$ factor, and has no $[3,-1]$~block.
Each negative column is paired with the degree $(1,1,1)$ column as in Lemma~\ref{lem:disguise}, but the two eliminations interact, and the intermediate fourfold is a smooth toric Fano fourfold rather than a complete intersection.
Let $[s_0:s_1],[t_0:t_1]\in\PP^1$ be the two projective coordinates.
After a linear change of the first 4 coordinates of $\PP^5$, the degree $(1,1,1)$ equation can be written
\[
p_2=s^tM(x)t, \qquad
M(x)=\begin{pmatrix}x_0&x_1\\x_2&x_3\end{pmatrix}.
\]
By Lemma~\ref{lem:disguise}, applied to the first $\PP^1$ factor with $(f_0,f_1)=M(x)t$, the third column imposes $M(x)t=0$, i.e. both rows of $M$ are proportional to $(-t_1,t_0)$:
\[
M(x)=\tilde w\,(-t_1,t_0), \qquad
\tilde w=(\tilde w_0,\tilde w_1)^t\in H^0\bigl(\OO(0,-1,1)^{\oplus2}\bigr).
\]
The locus $\{Mt=0\}\subset\PP^1_t\times\PP^5$ is the projective bundle $\PP_{\PP^1_t}(\OO^{\oplus2}\oplus\OO(-1)^{\oplus2})$ with fiber coordinates $(x_4,x_5,\tilde w_0,\tilde w_1)$, 
and on $\PP^1_s\times\{Mt=0\}$ the sections of $\OO(1,-1,1)$ form the 4-dimensional space spanned by $s_i\tilde w_j$.
The fourth column therefore is $s^t\mathsf A\tilde w=0$ for a $2\times2$ matrix $\mathsf A$, and for general (invertible) $\mathsf A$ the linear change $M\mapsto\mathsf A^tM$ of $x_0,\ldots,x_3$ normalizes it to $s^tM=0$.
The same conclusion holds if the columns are taken in the order of \cite{AAGGL}: the first column contributes no additional sections to the negative line bundles, since $H^1(\OO(-2,0,-2))$ on the ambient space and $H^1(\OO(0,-2,-2))$ on $\PP^1_s\times\{Mt=0\}$ vanish.
Therefore $M(x)=\lambda\binom{-s_1}{s_0}(-t_1,t_0)$, and together with the remaining 2 coordinates of $\PP^5$ this identifies the intermediate fourfold with
\begin{equation}
V_5=
\PP_{\PP^1\times\PP^1}
\bigl(\OO^{\oplus2}\oplus\OO(-1,-1)\bigr),
\end{equation}
where we use the convention of one-dimensional subspaces.
If $A_1,A_2$ are the ruling classes and $H=c_1(\OO_{V_5}(1))$, then $-K_{V_5}=3H+A_1+A_2$ is ample, so $V_5$ is a smooth toric Fano fourfold, with Cox coordinates $s$, $t$ of classes $A_1$, $A_2$, $x_4,x_5$ of class $H$, and $\lambda$ of class $H-A_1-A_2$.
The remaining degree $(1,1,3)$ equation restricts to a general anticanonical section: the sections of $\OO_{V_5}(-K_{V_5})$ are the monomials $\lambda^cx_4^ax_5^b$ with $a+b+c=3$ times forms of bidegree $(1+c,1+c)$ in $(s,t)$, 
and restricting a $(1,1,3)$-form and substituting the entries of $M(x)$ produces all of them, since every form of bidegree $(1+c,1+c)$ is a sum of products of a $(1,1)$-form and a $(c,c)$-form.
Hence we have
\begin{equation}
X_{5,+}\cong X_5\in|-K_{V_5}|.
\end{equation}

The cohomology ring is
\[
H^*(V_5,\ZZ)=\ZZ[H,A_1,A_2]/(A_1^2,A_2^2,H^2(H-A_1-A_2)), \qquad
\int_{V_5}H^2A_1A_2=1.
\]
Intersecting with $[X_5]=3H+A_1+A_2$, we obtain, in the basis $(H,A_1,A_2)$,
\begin{equation}
\mu_5(x,y,z)=8x^3+12x^2y+12x^2z+18xyz,
\end{equation}
\[
c_2(X_5)\cdot(H,A_1,A_2)=(56,24,24), \qquad e(X_5)=-156,
\]
and the rational even cohomology ring is
\begin{equation}
\mathcal R_5
=\QQ[H,A_1,A_2]/(A_1^2,A_2^2,9H^2-12H(A_1+A_2)+8A_1A_2).
\end{equation}
Since $X_5$ is an ample divisor in the smooth toric fourfold $V_5$, we have $\operatorname{Pic}(X_5)=\operatorname{Pic}(V_5)=\ZZ H\oplus\ZZ A_1\oplus\ZZ A_2$ by the Lefschetz theorem. 
Together with $e(X_5)=-156$, we obtain $(h^{1,1},h^{2,1})=(3,81)$.\footnote{
The dimension count $h^0(V_5,\OO_{V_5}(-K_{V_5}))-1-\dim\operatorname{Aut}(V_5)=100-1-18=81$ is compatible with this.}
The fibration $X_5\to\PP^1\times\PP^1$ has a plane cubic as general fiber. 
Its index is 3, since the fiber class is $A_1A_2$ and $(rH+a_1A_1+a_2A_2)\cdot A_1A_2=3r$ (Lemma~\ref{lem:index}).

%%%%%%%%%%%%%%%%%%%%%%%%%%%%%%%%%%%%%%%%%%%%%%%%%%%%%%%%%%%%%%%%%%%%%%%%%%%%%%%%%%%%%%%%%%%%%%
\section{Toric mirrors, periods, and indicial algebras}
\label{sec:mirrors}

In this section, we will describe the toric mirrors of the five examples and their fundamental periods. 
The toric phases of Examples 1, 2 and 4 are complete intersections in toric ambient varieties, and their mirrors are the Batyrev--Borisov/Hori--Vafa models, with mirror Hodge numbers given by Batyrev--Borisov duality. 
The toric phases of Examples 3 and 5, namely $X_{3,-}$ and $X_5$, are anticanonical hypersurfaces in smooth toric Fano fourfolds, so Batyrev's reflexive-polytope construction applies directly. 
For the sake of period computation, we also use the convenient Hori--Vafa equations 
\[
\sum_{\rho\in E_j}X_\rho=1, \qquad \prod_\rho X_\rho^{D_\rho\cdot\beta}=q^\beta, 
\]
which are the Batyrev--Borisov mirror Calabi--Yau threefolds restricted to the torus, 
with the coefficients absorbed into the parameters $q^\beta$ \cite{BB,HV}.  
We call them the torus presentation of the Batyrev--Borisov mirror.

\subsection{Example 1}
Let $(z_1,z_2,z_3)$ be the coordinates dual to $(H,T,G)$. For the ordinary CICY presentation of the toric phase $X_{1,-}$ (Proposition~\ref{prop:Rzero}), of the same type as \eqref{eq:X1cicy}, the torus presentation of the Batyrev--Borisov mirror is
\begin{equation}
a+b+d+e=1,\qquad
c+\frac{z_1}{abc}+\frac{z_2}{d}+\frac{z_3}{e}=1.
\end{equation}
Its fundamental period is
\begin{equation}
\Phi_1=
\sum_{n,m,l\ge0}
\frac{(2n+m+l)!^2}
{(n!)^4(m!)^2(l!)^2}
z_1^nz_2^mz_3^l.
\end{equation}
A crepant resolution of the mirror has Hodge numbers $(h^{1,1},h^{2,1})=(55,3)$ by Batyrev--Borisov duality. 
For this example, we use the period associated with the ordinary CICY presentation rather than the uniform period \eqref{eq:uniformperiod}; the two arise from different toric ambient presentations of the same threefold.

\subsection{Example 2}
Let $(z_1,z_2,z_3)$ be the coordinates dual to $(H,G,T)$, and put
\[
\mathcal K(a,b,c,d)
=\frac{(1+a)^2(1+b)^2(1+c)^2(1+d)^2}{a^2b^2cd}.
\]
The torus presentation of the Batyrev--Borisov mirror can be written as
\begin{equation}
1-z_1\mathcal K=0, \qquad
1-z_2ab-t-\frac{z_3}{t}=0.
\end{equation}
Its fundamental period is
\begin{equation}
\Phi_2=
\sum_{n,l\ge0}\sum_{m=0}^{2n}
\frac{(2n)!^4(m+2l)!}
{(n!)^4(m!)^3(l!)^2(2n-m)!^2}
z_1^nz_2^mz_3^l,
\end{equation}
which is the specialization of \eqref{eq:uniformperiod} for
$a=(2,0)$ and $r=(0,2)$.
By Theorem~\ref{thm:BBmirror}, a crepant resolution of the mirror has Hodge numbers $(h^{1,1},h^{2,1})=(31,3)$.

\subsection{Example 3}
Let $(z_1,z_2)$ be the coordinates dual to $(H,G)$. 
Since $X_{3,-}$ is an anticanonical hypersurface in the smooth toric Fano fourfold $V_3$, its Batyrev mirror is represented on the dense torus by
\begin{equation}
1-a-b-c-d-\frac{z_1}{cd}-\frac{z_2cd}{ab}=0.
\end{equation}
Its fundamental period is
\begin{equation}
\Phi_3=
\sum_{n\ge m\ge0}
\frac{(3n+m)!}
{n!(m!)^3(n-m)!^2}
z_1^nz_2^m,
\end{equation}
which is \eqref{eq:uniformperiod} for $a=1$ and $r=3$.
Since $(h^{1,1},h^{2,1})(X_{3,-})=(2,86)$, a crepant resolution of the mirror has Hodge numbers $(h^{1,1},h^{2,1})=(86,2)$ by Batyrev mirror symmetry.

\subsection{Example 4}\label{sec:mirror4}
Let $(z_1,z_2,w)$ be the coordinates dual to $(H_1,H_2,G)$. As in \cite[Lemma~3.1]{KanCompanion}, the Hori--Vafa equations associated with Proposition~\ref{prop:BBduality} reduce to the Laurent hypersurface
\begin{equation}
1-wab-\Bigl(\frac{z_1}{cd}+z_2\Bigr)
\frac{(1+a)(1+b)(1+c)(1+d)}{ab}=0
\qquad\text{in }(\CC^*)^4.
\end{equation}
Its fundamental period is
\begin{equation}
\Phi_4=
\sum_{n_1,n_2\ge0}\ \sum_{m=0}^{n_1+n_2}
\frac{(n_1+n_2)!^4\,(n_1+n_2+m)!}
{(n_1+n_2-m)!^2\,(m!)^3\,(n_1!)^3(n_2!)^3}
z_1^{n_1}z_2^{n_2}w^m,
\label{eq:period4}
\end{equation}
which is \eqref{eq:uniformperiod} for $a=r=(1,1)$.
By Theorem~\ref{thm:BBmirror}, a crepant resolution of the mirror has Hodge numbers $(h^{1,1},h^{2,1})=(42,3)$. 

On $z_1=z_2=0$, the $\mathcal R_4$-valued Frobenius series of Section~\ref{sec:frobdata}, after removing its exponential factor, becomes
\[
1+F^2\Li_2(w)-12[\mathrm{pt}]\,\Li_3(w),
\qquad
F=H_1+H_2-G,
\]
where $\Li_k(q)=\sum_{m\ge1}q^m/m^k$. Here $F^2$ is Poincar\'e dual to the sum of the 6 flopping curves. This is the analogue of the boundary series in \cite{KanCompanion}.

\subsection{Example 5}
Let $(z_1,z_2,z_3)$ be the coordinates dual to $(H,A_1,A_2)$. 
Since $X_5$ is an anticanonical hypersurface in the smooth toric Fano fourfold $V_5$, its Batyrev mirror is represented on the dense torus by
\begin{equation}
1-a-b-c-d-\frac{z_2d}{a}-\frac{z_3d}{b}-\frac{z_1}{cd}=0.
\end{equation}
Its fundamental period is
\begin{equation}
\Phi_5=
\sum_{\substack{n,m,l\ge0\\ n\ge m+l}}
\frac{(3n+m+l)!}
{(n!)^2(m!)^2(l!)^2(n-m-l)!}
z_1^nz_2^mz_3^l.
\end{equation}
Since $(h^{1,1},h^{2,1})(X_5)=(3,81)$, a crepant resolution of the mirror has Hodge numbers $(h^{1,1},h^{2,1})=(81,3)$ by Batyrev mirror symmetry.

%%%%%%%%%%%%%%%%%%%%%%%%%%%%%%%%%%%%%%%%%%%%%%%%%%%%%%%%%%%%%%%%%%%%%%%%%%%%%%%%%%%%%%%%%%%%%%%%%%%
\subsection{Picard--Fuchs systems and maximally unipotent boundary points}\label{sec:frobdata}

We will use the Frobenius method and indicial algebras as in \cite{KanCompanion}.
For Example~$i$, let $\mathbf z_i=(z_{i,1},\ldots,z_{i,r_i})$ be the variables of the fundamental period $\Phi_i$, 
and let $\mathbf D_i=(D_{i,1},\ldots,D_{i,r_i})$ be the corresponding divisor classes in $\mathcal R_i$, in the order of the preceding subsections. Write
\[
\Phi_i(\mathbf z_i)
=
\sum_{\mathbf n}c_i(\mathbf n)\mathbf z_i^{\mathbf n}.
\]
Replacing the factorials in $c_i$ by Gamma functions and shifting
$\mathbf n$ to $\mathbf n+\mathbf D_i$, we obtain the
$\mathcal R_i$-valued Frobenius series
\[
\mathcal I_i(\mathbf z_i)
=
\mathbf z_i^{\mathbf D_i}
\sum_{\mathbf n\in\ZZ_{\ge0}^{r_i}}
\frac{c_i(\mathbf n+\mathbf D_i)}
     {c_i(\mathbf D_i)}
\mathbf z_i^{\mathbf n}.
\]
Here the coefficients are understood by their Gamma-function continuation. In particular, terms outside the original summation range are retained; reciprocal Gamma factors then give nilpotent coefficients. 
The periods $\Phi_i$ are annihilated by the Picard--Fuchs operators listed in Appendix~\ref{app:operators}.

Let $I_{\mathrm{ind},i} \subset \CC[\rho_1,\ldots,\rho_{r_i}]$ be the ideal generated by the indicial polynomials of the Picard--Fuchs system at $\mathbf z_i=0$. 
We define the indicial algebras by 
\[
A_{\mathrm{ind},i}
=
\CC[\rho_1,\ldots,\rho_{r_i}]
/
I_{\mathrm{ind},i}.
\]
The terminology and the general construction are recalled in \cite[Section~4.2]{KanCompanion}.

For the five examples, we have isomorphisms
\begin{equation}
A_{\mathrm{ind},i}
\cong
\mathcal R_i\otimes_{\QQ}\CC,
\qquad
\rho_j\longmapsto D_{i,j}.
\label{eq:indicial-cohomology}
\end{equation}
Indeed, since $\mathcal I_i$ is annihilated by the Picard--Fuchs system, the lowest-order terms imply that every indicial polynomial vanishes after the substitution
$\rho_j=D_{i,j}$. Therefore \eqref{eq:indicial-cohomology} is a surjection.
On the other hand, the indicial ideals obtained from the operators in Appendix~\ref{app:operators} have Hilbert functions
$(1,3,3,1)$, $(1,3,3,1)$, $(1,2,2,1)$, $(1,3,3,1)$, $(1,3,3,1)$, respectively. 
These coincide with the Hilbert functions of the rings
$\mathcal R_i$, and hence the surjections are isomorphisms.
The Frobenius construction therefore gives $\dim_\QQ\mathcal R_i=8, 8, 6, 8, 8$ linearly independent solutions, respectively, which is the third Betti number $2h^{2,1}+2$ of the mirror threefold.

Under \eqref{eq:indicial-cohomology}, the logarithms of the local monodromies act, up to the factor $2\pi i$, by multiplication by the corresponding divisor classes. Choose an ample class $D$ of the toric phase $X_{i,-}$ (resp. $X_5$), and let $N$ be the corresponding linear combination of the monodromy logarithms.
Since $D$ is ample on a threefold, $D^3\ne0$, while $D^4=0$. 
Hence $N^3\ne0$, and $N^4=0$. 
Therefore the origin is a maximally unipotent boundary point in each of the five examples.

\vspace{2mm}

\paragraph{Acknowledgments.}
The author thanks Makoto Miura for valuable discussions.
Part of this work was inspired by interactions during the author's visit to National Taiwan University in spring 2026, whose hospitality is gratefully acknowledged.
This work is supported by JSPS Grants-in-Aid Kiban (C) 22K03296 and Kiban (A) 23H00083.

\paragraph{AI disclosure.}
ChatGPT was used for mathematical exploration, discussing proofs and editing the manuscript. 
The author verified the mathematical results and proofs and takes full responsibility for the manuscript.

\newpage

%%%%%%%%%%%%%%%%%%%%%%%%%%%%%%%%%%%%%%%%%%%%%%%%%%%%%%%%%%%%%%%%%%%%%%%%%%%%%%%%%%%%%%%%%%%%%%

\appendix
\section{Picard--Fuchs operators}
\label{app:operators}

Let $\theta_i=z_i\partial_{z_i}$ be the Euler operator. 
The operators below annihilate the fundamental periods of Section~\ref{sec:mirrors}.
After division by the coefficient, each annihilation is an identity of rational functions of the summation indices, which can be checked by a straightforward but tedious computation (cf. Appendix~\ref{app:operators4}); the finitely many boundary rows are checked directly.

\subsection{Example 1}
Let $D=2\theta_1+\theta_2+\theta_3$ and
\[
\begin{aligned}
P_2&=3\theta_2\theta_3-2\theta_1\theta_2+3\theta_3^2
+6\theta_1\theta_3-6\theta_1^2,\\
P_3&=3\theta_2\theta_3-2\theta_1\theta_3+3\theta_2^2
+6\theta_1\theta_2-6\theta_1^2.
\end{aligned}
\]
Then
\begin{align*}
\mathcal L^{(1)}_1&=\theta_2^2-z_2(D+1)^2,\\
\mathcal L^{(1)}_2&=\theta_3^2-z_3(D+1)^2,\\
\mathcal L^{(1)}_3&=2\theta_1^2-2\theta_1(\theta_2+\theta_3)+3\theta_2\theta_3
-8z_1(D+1)^2\\
&\quad-z_2\{P_2+3\theta_3-2\theta_1\}
-z_3\{P_3+3\theta_2-2\theta_1\}.
\end{align*}

\subsection{Example 2}
Let $K=2\theta_1-\theta_2$ and $D=\theta_2+2\theta_3$.
Then
\begin{align*}
\mathcal L^{(2)}_1&=\theta_3^2-z_3(D+1)(D+2),\\
\mathcal L^{(2)}_2&=\theta_2^2-2\theta_2\theta_3+4\theta_3^2
-8z_3(\theta_3+1)(D+1)-z_2K^2,\\
\mathcal L^{(2)}_3&=K^2
-16z_1(4\theta_1^2+4\theta_1\theta_2+4\theta_1
+3\theta_2^2+2\theta_2+1)\\
&\quad-16z_1z_2(8\theta_1-3\theta_2+1)(D+1).
\end{align*}

\subsection{Example 3}
Let $D=3\theta_1+\theta_2$ and $K=\theta_1-\theta_2$.
Then
\begin{align*}
\mathcal L^{(3)}_1&=9\theta_1^2-21\theta_1\theta_2+16\theta_2^2
-27z_1(D+1)(D+2)-16z_2K^2,\\
\mathcal L^{(3)}_2&=\theta_2^3-z_2(D+1)K^2.
\end{align*}

\subsection{Example 4}
\label{app:operators4}
Use $(z_1,z_2,w)$ with Euler operators $\theta_1,\theta_2,\theta_w$, and put $S=\theta_1+\theta_2$.
Then
\begin{align*}
\mathcal L^{(4)}_1={}&S^2-3S\theta_w+4\theta_w^2-4w(S-\theta_w)^2\\
&-z_1(\theta_1^2+5\theta_1\theta_2+10\theta_2^2+2\theta_1+5\theta_2+1)
 -z_2(10\theta_1^2+5\theta_1\theta_2+\theta_2^2+5\theta_1+2\theta_2+1),\\
\mathcal L^{(4)}_2={}&\theta_1^2-\theta_1\theta_2+\theta_2^2
-(z_1+z_2)\bigl((S+1)^2+3(S+1)\theta_w+5\theta_w^2\bigr)\\
&+(z_1+z_2)w\bigl(5\theta_w^2+(3-2S)\theta_w-(7S+2)(S+1)\bigr),\\
\mathcal L^{(4)}_3={}&(\theta_1-\theta_2)(3S-5\theta_w)
-\bigl(z_1(\theta_1+2\theta_2+1)-z_2(2\theta_1+\theta_2+1)\bigr)(3S+3+4\theta_w)\\
&+4w\bigl(z_1(\theta_1+2\theta_2+1)-z_2(2\theta_1+\theta_2+1)\bigr)(S+\theta_w+1),\\
\mathcal L^{(4)}_4={}&\theta_w^3-w(S+\theta_w+1)(S-\theta_w)^2 .
\end{align*}
The last operator is the $w$-recurrence of \eqref{eq:period4}.
The first 3 operators also annihilate $\Phi_4$ and the $\mathcal R_4$-valued Frobenius series, as we verify below.

We write $c(n_1,n_2,m)$ for the coefficients of \eqref{eq:period4} and $N=n_1+n_2$.
In terms of the Gamma function, we have the meromorphic identities
\[
\frac{c(n_1,n_2,m-1)}{c(n_1,n_2,m)}=\frac{m^3}{(N+m)(N-m+1)^2},\qquad
\frac{c(n_1-1,n_2,m)}{c(n_1,n_2,m)}=\frac{n_1^3(N-m)^2}{N^4(N+m)},
\]
and similarly for $n_2$.
The shifts in \eqref{eq:period4} used by the operators are products of these.
After division by $c(n_1,n_2,m)$, the coefficient of $z_1^{n_1}z_2^{n_2}w^m$ in $\mathcal L\Phi_4$ becomes a rational function of $(n_1,n_2,m)$.
For each of the 4 operators, after clearing denominators, a direct expansion shows that the resulting polynomial vanishes identically.
This proves the annihilation for all coefficients at which all the Gamma quotients involved are holomorphic, which excludes only the rows $n_1+n_2\le1$.
There the coefficients are finite in number, since $c(n_1,n_2,m)=0$ for $m>n_1+n_2$, and they are checked directly.

The same identities, with $(n_1,n_2,m)$ replaced by $(n_1+H_1,n_2+H_2,m+G)$, give the annihilation of the $\mathcal R_4$-valued series, again except in the rows $n_1+n_2\le1$.
In the row $n_1=n_2=0$, the series is the boundary series $\sum_{m\ge0}r_m w^m=1+F^2\Li_2(w)-12[\mathrm{pt}]\Li_3(w)$ of Section~\ref{sec:mirror4}, with $F=H_1+H_2-G$.
Only the terms of $\mathcal L^{(4)}_1$ without $z_1,z_2$ contribute to this row, and the check is the computation of \cite[Appendix~D.1]{KanCompanion}, using $H_iF^2=0$ and $GF^2=6[\mathrm{pt}]$.
The indicial parts of $\mathcal L^{(4)}_2$ and $\mathcal L^{(4)}_3$ annihilate $F^2$ and $[\mathrm{pt}]$ by the relations of $\mathcal R_4$, and the $w$-recurrence $\mathcal L^{(4)}_4$ is treated like the operator $\mathcal L_3$ there.
The 2 rows such that $n_1+n_2=1$ and $m=0$ are checked directly.

\subsection{Example 5}
Let $D=3\theta_1+\theta_2+\theta_3$ and $K=\theta_1-\theta_2-\theta_3$.
Then
\begin{align*}
\mathcal L^{(5)}_1&=\theta_2^2-z_2(D+1)K,\\
\mathcal L^{(5)}_2&=\theta_3^2-z_3(D+1)K,\\
\mathcal L^{(5)}_3&=9\theta_1^2-12\theta_1(\theta_2+\theta_3)+8\theta_2\theta_3
-27z_1(D+1)(D+2)\\
&\quad+4z_2(3\theta_1-2\theta_3)K
+4z_3(3\theta_1-2\theta_2)K.
\end{align*}

%-----------------------------------------------------------------------------

%%%%%%%%%%%%%%%%%%%%%%%%%%%%%%%%%%%%%%%%%%%%%%%%%%%%%%%%%%%%%%%%%%%%%%%%%%%%%%%%%%%%%%%%%%%%%%%%%%%%%%%%%

\par\noindent{\scshape \small
Department of Mathematics \\
Faculty of Science and Engineering\\
Waseda University \\
3-4-1 Okubo, Shinjuku, Tokyo, 169-8555, Japan}
\par\noindent{\ttfamily a\_kanazawa@waseda.jp}


\begin{thebibliography}{99}

\bibitem{AAGGL}
L.~B.~Anderson, F.~Apruzzi, X.~Gao, J.~Gray, and S.-J.~Lee, \emph{A new construction of Calabi--Yau manifolds: generalized CICYs}, Nucl. Phys. B \textbf{906} (2016), 441--496.

\bibitem{GG}
A.~Garbagnati and B.~van Geemen, \emph{A remark on generalized complete intersections}, Nucl. Phys. B \textbf{925} (2017), 135--143.

\bibitem{BH}
P.~Berglund and T.~H\"ubsch, \emph{On Calabi--Yau generalized complete intersections from Hirzebruch
varieties and novel K3-fibrations}, Adv. Theor. Math. Phys. \textbf{22} (2018), 261--303.

\bibitem{BHmirror}
P.~Berglund and T.~H\"ubsch, \emph{A generalized construction of Calabi--Yau models and mirror symmetry}, SciPost Phys. \textbf{4} (2018), 009.

\bibitem{BHTyurin}
P.~Berglund and T.~H\"ubsch, \emph{Hirzebruch surfaces, Tyurin degenerations and toric mirrors: bridging
generalized Calabi--Yau constructions}, Adv. Theor. Math. Phys. \textbf{26} (2022), 2541--2598.

\bibitem{JiaLin}
Q.~Jia and H.~Lin, \emph{Calabi--Yau generalized complete intersections and aspects of
cohomology of sheaves}, J. Math. Phys. \textbf{61} (2020), 052301.

\bibitem{LPS}
M.~Larfors, D.~Passaro, and R.~Schneider, \emph{Heterotic line bundle models on generalized complete intersection
Calabi--Yau manifolds}, JHEP \textbf{05} (2021), 105.

\bibitem{CuiGaoWang}
W.~Cui, X.~Gao, and J.~Wang, \emph{Machine learning on generalized complete intersection Calabi--Yau
manifolds}, Phys. Rev. D \textbf{107} (2023), 086004.

\bibitem{BB}
V.~V.~Batyrev and L.~A.~Borisov, \emph{On Calabi--Yau complete intersections in toric varieties}, in: Higher-dimensional complex varieties (Trento, 1994), de Gruyter, Berlin, 1996, 39--66.

\bibitem{Batyrev}
V.~V.~Batyrev, \emph{Dual polyhedra and mirror symmetry for Calabi--Yau hypersurfaces in toric varieties}, J. Algebraic Geom. \textbf{3} (1994), 493--535.

\bibitem{BatyrevBetti}
V.~V.~Batyrev, \emph{Birational Calabi--Yau $n$-folds have equal Betti numbers}, in: New trends in algebraic geometry (Warwick, 1996), London Math. Soc. Lecture Note Ser. \textbf{264}, Cambridge Univ. Press, 1999, 1--11.

\bibitem{HV}
K.~Hori and C.~Vafa, \emph{Mirror symmetry}, arXiv:hep-th/0002222.

\bibitem{KanCompanion}
A.~Kanazawa, \emph{Mirror symmetry for gCICY threefolds (I): a $[3,-1]$~block and a non-Gorenstein toric phase}, arXiv:2610.03688. 

\bibitem{KanIII}
A.~Kanazawa, \emph{gCICY threefolds with a $[4,-2]$~block: Grassmannian and toric-join phases}, in preparation.

\bibitem{KM}
J.~Koll\'ar and S.~Mori, \emph{Birational Geometry of Algebraic Varieties}, Cambridge Tracts in Math. \textbf{134}, Cambridge Univ. Press, 1998.

\bibitem{LLW}
Y.-P.~Lee, H.-W.~Lin, and C.-L.~Wang, \emph{Flops, motives and invariance of quantum rings}, Ann. of Math. (2) \textbf{172} (2010), 243--290.


\bibitem{Givental}
A.~Givental, \emph{A mirror theorem for toric complete intersections}, in: Topological field theory, primitive forms and related topics (Kyoto, 1996), Progr. Math. \textbf{160}, Birkh\"auser, Boston, 1998, 141--175.


\bibitem{HKTY}
S.~Hosono, A.~Klemm, S.~Theisen, and S.-T.~Yau, \emph{Mirror symmetry, mirror map and applications to complete intersection
Calabi--Yau spaces}, Nucl. Phys. B \textbf{433} (1995), 501--552.

\end{thebibliography}
\end{document}